\documentclass[11pt]{amsart}

\usepackage{amsmath,amssymb}

\usepackage{amssymb}
\usepackage[utf8]{inputenc}
\usepackage{amsfonts}
\usepackage{amsmath}
\usepackage{amsthm}
\usepackage{dsfont}
\usepackage[dvips]{graphicx}
\usepackage{enumerate}
\usepackage[all]{xy}
\usepackage{pstricks}
\usepackage{stackrel}

\theoremstyle{plain} \newtheorem{teo}{Theorem}[section]
\theoremstyle{definition} \newtheorem{defi}[teo]{Definition}

\theoremstyle{plain} \newtheorem{prop}[teo]{Proposition}
\theoremstyle{plain} \newtheorem{lema}[teo]{Lemma}
\theoremstyle{plain} \newtheorem{cor}[teo]{Corollary}
\theoremstyle{remark} \newtheorem{obs}[teo]{Remark}
\theoremstyle{remark} \newtheorem{ej}[teo]{Example}

\newcommand{\tor}{\mbox{Tor}}

\newcommand{\Hom}{\mbox{Hom}}
\newcommand{\mo}{\mbox{mod}}

\newcommand{\gd}{\mbox{gldim}}
\newcommand{\pd}{\mbox{pd}}

\newcommand{\ann}{\mbox{ann}}
\newcommand{\tp}{\mbox{top}}
\newcommand{\rad}{\mbox{rad}}
\newcommand{\End}{\mbox{End}}

\title[On the global dimension of the endomorphism algebra of a $\tau$-tilting module]
{On the global dimension of the endomorphism algebra of a $\tau$-tilting module}

\author[Suarez Pamela ]{Pamela Suarez}
\address{Centro Marplatense de Investigaciones Matem\'aticas (CEMIM), Facultad de Ciencias Exactas y
Naturales, Funes 3350, Universidad Nacional de Mar del Plata, 7600 Mar del
Plata, Argentina}
\email{pamelaysuarez@gmail.com}
\keywords{$\tau$-tilting modules, Endomorphism algebra, Global dimension, Annihilator}
\subjclass[2010]{16G20, 16E10}

\begin{document}
\sloppy
\maketitle
\begin{abstract}
We find a relationship between the global dimension of an algebra $A$ and the global dimension of the endomorphism algebra of a $\tau$-tilting module, when $A$ is of finite global dimension.  We show that, in general, the global dimension of the endomorphism algebra is not always finite. For monomial algebras and special biserial algebras of global dimension two, we prove that the global dimension of the endomorphism algebra of any $\tau$-tilting module is always finite. Moreover, for special biserial algebras, we give an explicit bound.
\end{abstract}

\section*{Introduction}

Let $A$ be a finite dimensional algebra over an algebraically closed field $k$. The $\tau$-tilting theory was introduced by T. Adachi, O. Iyama and I. Reiten in \cite{AIR}. In such a work, the mentioned authors developed a generalization of the classical tilting theory where the process of mutation is always possible. The basic idea of a mutation is to replace an indecomposable direct summand of a tilting module by another indecomposable module in order to obtain a new tilting module. For the process of mutation to be always possible the authors considered the notion of $\tau$-rigid modules introduced by M. Auslander and S. Smalø in \cite{AS} and, moreover, they also considered the notion of support tilting module introduced by C. Ingalls and H. Thomas in \cite{IT}. In this wider class of modules it is possible to model the process of mutation inspired by the cluster tilting theory. In this way we get that an almost complete support $\tau$-tilting module over a finite dimensional algebra has exactly two complements, and therefore, the mutation is always possible.  Support $\tau$-tilting modules are very closely related with  functorially finite torsion classes  in $\mo\,A$, with cluster tilting theory and with silting theory, between others.

Let $M$ be a (classical) tilting module in the category $\mo\,A$. Then it is  well-known  that ${\gd\,\mbox{End}_AM\leq \gd\,A +1}$, see \cite[Chapter VI, Theorem 4.2]{ASS}. In  general, if we consider $N$ a tilting module of finite projective dimension, then we know that  ${\gd\,\mbox{End}_AN\leq \gd\,A +\mbox{pd}_AN}$, see \cite[Corollary 2.4]{Miya}. In this paper we investigate how to extend these facts to the context of $\tau$-tilting modules.

Given $T$ a $\tau$-tilting $A$-module and $B=\End_AT$, we are mainly interested to compare the global dimension of $A$  and the global dimension  of $B$. More precisely, if $A$ has finite global dimension then we prove the following result.

\vspace{0.1in}
\noindent \textbf{Theorem A.}
\textit{Let $A$ be an algebra of finite global dimension, $T$ be a $\tau$-tilting $A$-module and $B=\emph{End}_AT$. Then $\emph{gldim}\, A \leq \emph{gldim}\, B + \emph{pd}_A(A/\emph{ann}\, T) + 1$, where $\emph{ann}\, T$ is the annihilator of $T$}.
\vspace{0.1in}

\indent In case we deal with (classical) tilting modules or with tilting modules of finite projective dimension, if $A$ is of finite global dimension then the endomorphism algebra of such modules is always of finite global dimension. In general, this is not the case for $\tau$-tilting modules. For example, for $n\geq 3$, we find a family of algebras $A_n$ of global dimension $n$ and  $\tau$-tilting $A_n$-modules $T_n$ such that $\gd\mbox{End}_{A_n}T_n=\infty$, see Example \ref{infinita}.

A natural question is to know if  the endomorphism algebra of a $\tau$-tilting module over an algebra of global dimension two has always finite global dimension. In order to give an answer to the above question we consider two families of algebras; the monomial  and the special biserial algebras. The main key to give a solution to this problem is to study the generators of the annihilator of a $\tau$-tilting module. We show that the generators of the annihilator of a $\tau$-tilting module are related with the relations of the algebra. More precisely, we prove Theorem B.

\vspace{0.1in}
\noindent \textbf{Theorem B.}
\textit{Let $A=kQ/I$ be an algebra, $T$ be a $\tau$-tilting $A$-module and ${\rho \in \emph{ann}\, T}$ be  a non-zero path of length greater than or equal to two such that no proper subpath of $\rho$ belongs to $\emph{ann}\,T$. Then one of the following conditions hold:
\begin{enumerate}[i)]
\item[i)] There exists a non-zero path $\gamma$ such that $\gamma\rho $ is a zero-relation in $A$, or
\item[ii)] There exist non-zero paths $\gamma_i, \gamma$ and non-zero scalars  $\lambda_i\in k$ such that ${\gamma\rho + \sum\limits_{i\in I}\lambda_i\gamma_i}$ is a minimal relation in $A$, where $I$ is a finite set of indices.
\end{enumerate}}
\vspace{0.1in}

For monomial algebras of global dimension two we prove the following result.

\vspace{0.1in}
\noindent \textbf{Theorem C.}
\textit{Let $A=kQ/I$ be a monomial algebra of global dimension two, $T$ be a $\tau$-tilting $A$-module and $B=\emph{End}_A\,T$. Then $\emph{gldim}\, B < \infty$.}
\vspace{0.1in}

Another natural question is to know if there exists an explicit bound for the $\gd\,\mbox{End}_AT$ where $A$ is a monomial algebra. In Example \ref{ejemplomonomialsincota}, we show that this is not the case. On the other hand, for special biserial algebras  we also prove that the global dimension of the endomorphism algebra of a $\tau$-tilting module over these algebras is always finite and, moreover, we give an explicit bound, proving Theorem $D$.

\vspace{0.1in}
\noindent \textbf{Theorem D.}
\textit{Let $A=kQ/I$ be a special biserial algebra such that ${\emph{gldim}\,A=2}$, $T$ be a $\tau$-tilting $A$-module and $B=\emph{End}_AT$. Then $\emph{gldim}\,B\leq 5$.}
\vspace{0.1in}

This paper is organized as follows. In the first section we fix some notation and present some preliminaries results. Section $2$ is dedicated to prove Theorem $A$ and we prove that the endomorphism algebra of a $\tau$-tilting module is not necessarily of finite global dimension. In section $3$, we study the annihilator of a $\tau$-tilting module and  we prove Theorem $C$ and Theorem $D$.

\section{Preliminaries}

Throughout this paper, all algebras are basic connected finite dimensional algebras over an algebraically closed field $k$. For an algebra $A$ we denote by $\mo\,A$ the category of finitely generated right $A$-modules.
\subsection{}


A \textbf{quiver} $Q$ is a quadruple $(Q_0,Q_1,s,t)$ where $Q_0$ is the set of points of $Q$, $Q_1$ is the set of arrows of $Q$, and $s,t$ are functions from $Q_1$ to $Q_0$ which give, respectively, the source $s(\alpha)$ and the target $e(\alpha)$ of a given arrow $\alpha$.

A {\bf path} in a quiver $Q$ is a sequence $\gamma =\alpha _1...\alpha _n$ such that $e(\alpha
_i)=s(\alpha _{i+1})$ for $1\leq i<n,$ as well as the symbol
$e_i$ for $i\in Q_0.$ The paths $e_i$ are called { \bf
trivial paths} and $s(e_i)=e(e_i)=i.$ For a non-trivial path
 $\gamma =\alpha _1...\alpha _n$, we have that $s(\gamma
)=s(\alpha _1)$ and $e(\gamma )=e(\alpha _n).$

A path $\gamma$ of \textbf{length} $l\geq 1$ with source $a\in Q_0$ and target $b\in Q_0$ is a sequence of $l$ arrows $\alpha _1,...,\alpha _l$ such that $s(\alpha _1)=a$, $e(\alpha _l)=b$ and $s(\alpha _{i+1})=e(\alpha _i)$, for $1\leq i \leq l-1$. If $e_i$ is a trivial path, then the length of $e_i$ is zero.

Given an algebra $A$, there exists a unique quiver $Q_A$ and (at least) a surjective algebra morphism $\eta:kQ_A\rightarrow A$, where $kQ_A$ denotes the path algebra of $Q_A$. Setting $I=\mbox{ker}\,\eta$, we then have $A\cong kQ_A/I$. The morphism $\eta$ is called a \textbf{presentation} of $A$, and $A$ is said to be given by the bound quiver $(Q_A,I)$. The ideal $I$ is generated by a finite set of relations. A relation in $Q_A$ from a vertex $x$ to a vertex $y$ is a linear combination $\rho=\sum\limits_{i=1}^nc_iw_i$, where  $c_i\in k\setminus \{0\} $ is non-zero for $i=1,\dots, n$, and for all $i$, $w_i$ are paths of length at least two from $x$ to $y$.

A relation  $\rho\in I$ is said to be \textbf{minimum} if $\rho=\sum\limits_{i}\beta_i\rho_i\gamma_i$, where $\rho_i$ is a relation for each $i$, then there exists $i$ such that $\beta_i, \,\gamma_i\in k$. In this work, the word relation means minimum relation in this sense.

A relation $\rho=\sum\limits_{i=1}^nc_iw_i$ is \textbf{monomial} or a \textbf{zero-relation} if $m=1$ and
\textbf{ minimal} if $m\geq 2$ and, for every non-empty and proper subset $J$ of $\{1,\dots,m\}$, we have that $\sum\limits_{j\in J}c_jw_j\notin I$. Moreover, a pair $(p,q)$ of non-zero paths $p$ and $q$ from a vertex  $a$ to a vertex $b$ is called a \textbf{binomial relation}  if $\lambda p+\mu q \in I$, with $\lambda,\mu\in k-\{0\}$. We say that $p$ and $q$ are the maximal subpaths of $(p,q)$.

Given a finite quiver $Q$, we denote by  $M=(M_a,\varphi_{\alpha})_{(a\in Q_0,\, \alpha\in Q_1)}$ a representation of $Q$. For a detail account on representation of quivers we refer the reader to \cite[Chapter III, Lemma 2.1]{ASS}.

\subsection{}
Recall that an  algebra $A=kQ/I$ is said to be \textbf{monomial} if the ideal $I$ is generated by paths. In \cite{GHZ}, E. Green, D. Happel and D. Zacharia gave an algorithm for computing projective resolutions of the simple modules over monomial algebras. In particular, the authors gave a bound for the global dimension of a monomial algebra in terms of the relations.

We recall here some definitions and results that shall be useful for our purposes.  We refer the reader to \cite{GHZ} for a detail account on this topics.

Let $(B,\rho)$ be a directed path with a minimal set of relations $\rho$. Given two vertices $u,w$ in $B$ we say that $u<w$ if there exists a directed path from $v$ to $w$.

\begin{defi}\cite{GHZ}
Let $v_0$ be a vertex of a directed path $(B,\rho)$. We define inductively the \textbf{associated sequence of relations  $\mathcal{S}$ of  $v_0$} (along $B$) as follows: if there is no $r\in\rho$ such that $s(r)=v_0$, then $\mathcal{S}=\emptyset$. Assume there is $r_1\in\rho$ such that $s(r_1)=v_0$.  Let $r_2$ be the relation in $\rho$ (if it exists) such that $s(r_1)<s(r_2)<e(r_1)$ where $s(r_2)$ is minimal satisfying this inequality. Assume that we have constructed $r_1, r_2,\dots,r_i$. Let
\[R_{i+1}=\{r\in\rho \,\mid\, e(r_{i-1})\leq s(r) < e(r_i)\}.\]
If $R_{i+1}\neq \emptyset$, let $r_{i+1}$ be such that $s(r_{i+1})$ is minimal with $r_{i+1}\in R_{i+1}$. Then $\mathcal{S}$ is the sequence  $r_1,r_2,\dots,r_N$ where $N$ is either $\infty$ if $R_i\neq \emptyset$ for all $i$, or $N$ is such that $R_{N+1}=\emptyset$ but $R_i\neq \emptyset$.
\end{defi}

\begin{teo}\cite[Theorem 1.2]{GHZ}
\label{arbolresproy}
Let $(B,\rho)$ be a directed path with minimal set of relations $\rho$. Let $v_0,v_1$ be vertices of $B$ such that there exists an arrow from $v_0$ to $v_1$. Let $\mathcal{S}=\{r_i\}_{i=1}^N$ be the associated sequence of relations of $v_0$ and $x_i=e(r_i)$.
\begin{enumerate}
  \item[a)] If $\mathcal{S}=\emptyset$, then $0\rightarrow P(v_1)\rightarrow P(v_0)\rightarrow S(v_0)\rightarrow 0$ is the minimal projective presentation of $S(v_0)$.
  \item[b)] Suppose $\mathcal{S}\neq \emptyset$. If $N<\infty$, then
  \[0\rightarrow P(x_N)\rightarrow P(x_{N-1})\rightarrow \dots \rightarrow P(x_1)\rightarrow P(v_1)\rightarrow P(v_0)\rightarrow S(v_0)\rightarrow 0\]
  is the minimal projective presentation of $S(v_0)$, and if $N=\infty$ then $\emph{pd}_BS(v_0)=\infty$ and its minimal projective resolution is
   \[\dots\rightarrow P(x_i)\rightarrow P(x_{i-1})\rightarrow \dots \rightarrow P(x_1)\rightarrow P(v_1)\rightarrow P(v_0)\rightarrow S(v_0)\rightarrow 0\]
\end{enumerate}
\end{teo}

Using the above result  and  the topological universal cover of a quiver,  the authors gave an algorithm to compute the projective resolutions of simple modules over monomial algebras, see  \cite[Theorem 2.3]{GHZ}. As a consequence of this fact, they proved the following proposition that  leads to a certain type of periodicity in the projective resolution of the simple $A$-modules of infinite projective dimension.


\begin{prop}\cite[Corollary 2.5]{GHZ}\label{propciclo}
\label{ciclosmonomial}
Let $A$ be a monomial algebra. Assume that $S$ is a simple $A$-module such that $\emph{pd}_AS=\infty$. Let
\[\dots\stackrel{f_n}\rightarrow A_{n-1} \rightarrow \dots \rightarrow A_1 \stackrel{f_1}\rightarrow A_0\stackrel{f_0}\rightarrow S \rightarrow 0 \]
be a minimal projective resolution of $S$. Then, there exists a sequence of  indecomposable projective $A$-modules $P_1,P_2,\dots,P_t$ and a positive integer  $l$ such that $P_i$ is a direct summand of $A_{m+l+i}$ for all $m\geq0$ and for $i=1,\dots,t$.
\end{prop}

\begin{obs}\label{obsciclo}
Consider $(\widetilde{Q},\widetilde{\rho})$ the universal cover of $(Q,\rho)$ and $\Phi:(\widetilde{Q},\widetilde{\rho}) \rightarrow (Q,\rho)$. In the proof of the Proposition \ref{ciclosmonomial}, the authors show that if $v$ is a vertex in $Q$ such that the simple module $S(v)$ has infinite projective dimension, then it is possible to construct an infinite sequence of minimal relations $r_1^*,r_2^*,\dots, r_n^*, \dots$ associated to $v^*$ in $(\widetilde{Q},\widetilde{\rho})$, with $\Phi(v^*)=v$ and $a,b$ integers with $1\leq a <b$, such that $r_i^*=r_i$ for $i=1,\dots, b+1$ and $\Phi(r^*_{lb+j})=\Phi(r_{a+j})$ for $l\geq 1$ and $0\leq j\leq b-a-1$. As a  consequence, there exists a sequence of relations associated to $v$ in $(Q,\rho)$ that winds up around a fixed cycle. Furthermore, since ${e(r_{i-1}^*)<s(r_{i+1}^*)<e(r_i^*)}$ in ${(\widetilde{Q},\widetilde{\rho})}$, there exist non-zero paths $\gamma_j$, for $j=1,\dots, n$,  such that $r_j=\gamma_j\gamma_{j+1}$ for $j=1,\dots,n-1$ and $r_n=\gamma_n\widetilde{\gamma}_1$, where $\widetilde{\gamma}_1$ is a subpath of $\gamma_1$.
\end{obs}

\subsection{}
We recall the definition of special biserial algebra.
\begin{defi}
An algebra  $A$ is \textbf{special biserial} if $A\cong kQ/I$ with $(Q,I)$ a quiver with relations satisfying the following conditions:
\begin{enumerate}
  \item[a)] Each vertex of $Q$ is  start-point or end-point of at most two arrows.
  \item[b)] For an arrow  $\alpha$, there is at most one arrow  $\beta$ such that $\alpha\beta\notin I$ and at most one arrow  $\gamma$ such that $\gamma\alpha\notin I$.
\end{enumerate}
\end{defi}

Next, we fix some notation and terminology that shall be used in this work.

Let $A=kQ/I$ be a special biserial algebra. A walk $w$ is called \textbf{reduced} if either $w$ is trivial or $w=c_1\dots c_n$ such that $c_i$ or $c_i^{-1}$ is an arrow and $c_{i+1}\neq c_i^{-1}$ for all $1\leq i \leq n$.   A reduced walk  $w$ in $Q$ is called a \textbf{string} if each path contained in $w$ is neither a zero-relation nor a maximal subpath of a binomial relation.

Let $A=kQ/I$ be a special biserial algebra. Let $w$ be a string in $(Q,I)$. We denote by $M(w)$ the string module determined by $w$. Recall that if $w$ is the trivial path at a vertex $a$, then $M(w)$ is the simple module at $a$. Otherwise $w=c_1c_2\dots c_n$ where $n\geq 1$ and $c_i$ or $c_i^{-1}$ is an arrow. For $1\leq i \leq n$, let $U_i=k$; and for $1\leq i \leq n$ we denote by $U_{c_i}$ the identity map sending $x\in U_i$ to $x\in U_{i+1}$ if $c_i$ is an arrow and otherwise the identity map sending $x\in U_{i+1}$ to $x\in U_{i}$. For a vertex $a$, if $a$ appears in $w$, then $M(w)_a$ is the direct sum of  spaces $U_i$ with $i$ such that $s(c_i)=a$ or $i=n+1$ and $e(c_n)=a$; otherwise $M(w)_a=0$. For an arrow $\alpha$, if $\alpha$ appears in $w$, then $M(w)_{\alpha}$ is the direct sum of the maps $U_{c_i}$   such that $c_i=\alpha$ or $c_i^{-1}=\alpha$; otherwise $M(w)_{\alpha}$ is the zero map. In some cases, if we deal with a projective module we write $P(w)$ instead of $M(w)$.


The following theorem of \cite{HL} gives a combinatorial characterization of the special biserial algebras such that the global dimension is at most two.

\begin{teo}\cite[Theorem 3.4]{HL}
\label{SBgldim}
Let $A=kQ/I$ be a special biserial algebra. Then the global dimension of $A$ is at most two if and only if $(Q,I)$ satisfies the following properties:
\begin{enumerate}
  \item[(GD1)] The start-point of a binomial relation does not lie in another different binomial relation.
  \item[(GD2)] There is no path of the form $p_1p_2p_3$, where $p_1,p_2,p_3$ are non-trivial paths such that $p_2$ is a string and $p_1p_2,\,p_2p_3$ are the only zero-relations contained in the path.
  \item[(GD3)] Let $(\alpha p \beta, \gamma q \delta)$ be a binomial relation, where $\alpha, \beta, \gamma, \delta$ are some arrows and $p,q$ are  paths. If $u$ is a non-zero path with $e(u)=s(\alpha)$, then either $u\alpha p$ or $u\gamma q$ is non-zero. Dually if $v$ is a non-zero path with $s(v)=e(\beta)$, then either $p\beta v$ or $q\delta v$ is non-zero.
\end{enumerate}
\end{teo}

\subsection{}

We recall that an $A$-module $M$ is said to be $\tau$-rigid if $\Hom_A(M,\tau M)=0$, where $\tau$ is the Auslander-Reiten translation.

\begin{defi}\cite[Definition 0.1]{AIR}\label{definicion support}
Let $M$ be a $\tau$-rigid module. We say that $M$ is $\tau$~-~tilting  if $M$ is $\tau$-rigid and $|M|=|A|$.
\end{defi}

By the Auslander-Reiten formula \cite[Chapter IV Theorem 2.13]{ASS} every rigid module is $\tau$-rigid and every tilting module is $\tau$-tilting. Moreover, it follows from \cite[Proposition 2.2]{AIR} that any $\tau$-tilting $A$-module is a tilting $A/\ann\,T$-module, where $\ann\,T$ denotes the annihilator of $T$.

We recall the following lemma of \cite{AIR} that shall be  useful.
%
\begin{lema}\cite[Proposition 2.4]{AIR}\label{resproytau}
Let $A$ be a finite dimensional algebra. Let $X$ be in $\emph{mod}\,A$ with a projective presentation $P_1\stackrel{p}\rightarrow P_0\rightarrow X\rightarrow 0$. For $Y\in \emph{mod}\,A$, if  $\emph{Hom}_A(p,Y)$ is surjective, then $\emph{Hom}_A(Y,\tau X)=0$. Moreover, the converse holds if the projective presentation is minimal.
\end{lema}

\section{A general bound}\label{seccionA}
We start studying homological relations between $A$ and $A/\mbox{ann}\,T$, or more generally between $A$ and $A/J$ with $J$ a nilpotent ideal of $A$.
\subsection{Change of rings functors}
Let $A$ be an algebra and $J$ be a nilpotent ideal of $A$. Consider the algebra morphism  $\varphi:A\rightarrow A/J$ given by $a\mapsto \overline{a}=a+J$. This morphism induces an exact functor $U:\mo\,A/J\rightarrow \mo\,A$. More precisely,  for each $A/J$-module $M$  the functor $U$ gives an  $A$-module structure. Thus, we get a natural immersion of  $\mo\,A/J$ in $\mo\,A$.

On the other hand, we consider the functor $F:\mo\,A\rightarrow \mo\,A/J$, given by ${N\mapsto F(N)=N\otimes_A A/J\cong N/NJ}$. It is well-known that $(F,U)$ is an adjoint pair of functors and if $P$ is a projective $A$-module then $A/J\otimes P$ is a projective $A/J$-module.

Since $J$ is a nilpotent ideal, then  $J\subset \mbox{rad}\,A$. Thus, ${\mbox{rad}(A/J)\cong \mbox{rad}\,A/J}$.
\begin{lema}\label{cubproy1}
Let $A$ be an algebra and $J$ a nilpotent ideal of $A$. If $f:P\rightarrow M$ is a projective cover in $\emph{mod}\,A$, then ${f\otimes 1:P\otimes_A A/J\rightarrow M\otimes_AA/J}$ is a projective cover in $\emph{mod}\,A/J$.
\end{lema}
\begin{proof}
It is clear that $f\otimes 1$ is an epimorphism and $A/J\otimes_A P$ is a projective  $A/J$~-module. It is only left to prove that $\mbox{top}(A/J\otimes_A P)=\mbox{top}(A/J\otimes_A M)$.

\begin{eqnarray*}
  \tp(P\otimes_A A/J) &\cong& \tp(P/PJ) \\
   &\cong& (P/PJ)/\rad(P/PJ) \\
   &\cong& (P/PJ)/(P/PJ\rad(A/J)) \\
   &\cong& (P/PJ)/((P/PJ)\rad A/J) \\
   &\cong& (P/PJ)/(P\rad A)/PJ) \\
   &\cong& P/\rad P \\
   &\cong& \tp P \\
   &\cong& \tp M \\
   &\cong& \tp( M\otimes_A A/J)
\end{eqnarray*}
Hence, by  \cite[Chapter 1, Lemma 5.6]{ASS} we have that $ 1\otimes f:A/J\otimes_A P\rightarrow A/J\otimes_AM$ is a projective cover in $\mo\,A/J$.
\end{proof}

The following lemma is the main key to establish a relationship between the global dimension of $A$ and the global dimension of $A/J$.

\begin{lema}\label{lempd}
Let $A$ be an algebra and $J$ a nilpotent ideal of $A$. Let $M$ be an $A$-module such that $\emph{Tor}_m^A(A/J,M)=0$ for all $m>0$. Then $\emph{pd}_AM=\emph{pd}_{A/J}(M/MJ)$.
\end{lema}
\begin{proof}
Consider $M\in \mo\,A$ with $\mbox{pd}_AM=n$. Then there exists a minimal projective resolution of $M$ as follows:
\begin{equation}\label{rp1}
  0\rightarrow P_n\stackrel{d_n}\rightarrow P_{n-1}\stackrel{d_{n-1}}\rightarrow \dots \rightarrow P_1\stackrel{d_1}\rightarrow P_0\stackrel{d_0}\rightarrow M\rightarrow 0.
\end{equation}

Since $\tor_m^A(A/J,M)=0$ for all $m>0$, applying the functor $\-- \otimes_AA/J$ to (\ref{rp1}), we obtain the exact sequence
\begin{equation}\label{rp2}
  0\rightarrow P_n\otimes_AA/J\stackrel{d_n\otimes_A1}\rightarrow  \dots \rightarrow P_1\otimes_AA/J\stackrel{d_1\otimes_A 1}\rightarrow P_0\otimes_AA/J\stackrel{d_0\otimes_A1}\rightarrow M\otimes_AA/J\rightarrow 0.
\end{equation}

On the other hand, since $\tor_1^A(A/J,\Omega_m(M))\cong\tor_{m+1}^A(A/J,M)=0$ we have that $\Omega_m( M\otimes_AA/J)\cong \Omega_m(M)\otimes A/J $. By Lemma \ref{cubproy1}, we get that ${P_m\otimes_AA/J\rightarrow \Omega_m(M\otimes_AA/J)}$ is a projective cover. Thus, the exact sequence in (\ref{rp2}) is a minimal projective resolution of $M\otimes_AA/J$.  Therefore, ${\mbox{pd}_AM=\mbox{pd}_{A/J}(M/MJ)}$.
\end{proof}

\begin{prop}\label{nilpotente}
Let $A$ be an algebra of finite global dimension and $J$ be a nilpotent ideal of $A$. Then $\emph{gldim}\, A\leq \emph{gldim}\, (A/J)+\emph{pd} _A(A/J)$.
\end{prop}
\begin{proof}
Let $M\in \mo\,A$ and $n=\pd_A(A/J)$. Consider the exact sequence
\[0\rightarrow M_n\rightarrow P_{n-1}\stackrel{d_{n-1}}\rightarrow \dots \rightarrow P_1\stackrel{d_1}\rightarrow P_0\stackrel{d_0} \rightarrow M \rightarrow 0\]
with each $P_i$ a projective $A$-module for $i=1,\ldots, n-1$ and $M_n=\mbox{ker}\,d_{n-1}$.

Then,  $\tor_k^A(A/J,M_n)\cong \tor^A_{k+n}(A/J,M)=0$ because ${n=\pd_A(A/J)}$. By Lemma \ref{lempd}, we know that $\mbox{pd}_AM_n=\mbox{pd}_{A/J}(M_n/M_nJ)$. Thus, we get that
\begin{eqnarray*}
  \pd_AM &\leq & n+ \mbox{pd}_{A/J}(M_n/M_nJ)\\
   &\leq& \pd_A(A/J)+\gd (A/J).
\end{eqnarray*}
Therefore, $\gd\, A\leq \gd\, (A/J)+\pd _A(A/J)$.
\end{proof}

\subsection{The endomorphism algebra of a $\tau$-tilting module}

Given $T$ a $\tau$-tilting $A$-~module, since $T$ is sincere then $\mbox{ann}\,T$ is a nilpotent ideal. As an application of  Proposition \ref{nilpotente} we obtain a relationship between the global dimension of $A$  and the global dimension of $B$, where $B=\End_A T$.

\begin{obs}
Given $T$ a $\tau$-tilting $A$-module, since $T$ is a  $(A/\mbox{ann}\, T)$-~module, then ${B=\mbox{End}_AT\cong \mbox{End}_{\small{A/\mbox{ann}\,T}}T}$. Moreover, since $T$ is a tilting $(A/\mbox{ann}\, T)$-module we have that if $\mbox{gldim}\,  A/\mbox{ann}\, T <\infty$ then
\[|\mbox{gldim}\, A/\mbox{ann}\, T-\mbox{gldim} \,B|\leq 1.\]
Therefore, $\mbox{gldim}\, B = \infty$ if and only if $\mbox{gldim}\, A/\mbox{ann}\, T = \infty$.
\end{obs}

As an application of Lemma \ref{nilpotente}, we obtain Theorem A.

\begin{teo}\label{cotagd1}
Let $A$ be an algebra of finite global dimension, $T$ be a $\tau$-tilting $A$-module and $B=\emph{End}_AT$. Then
\[\emph{gldim}\, A \leq \emph{gldim}\, B + \emph{pd}_A(A/\emph{ann}\, T) + 1.\]
\end{teo}
\begin{proof}
Since $\ann\, T$ is a nilpotent ideal, it follows from Lemma \ref{nilpotente} that ${\gd\, A\leq \gd\, A/\ann\, T+\pd _A(A/\ann\, T)}$.
If $\gd\, A/\ann\, T=\infty$, then $\gd\,B=\infty$ and we have nothing to prove.
Otherwise, ${\gd\, A/\ann\, T \leq \gd\, B + 1}$. Then,
\begin{eqnarray*}
  \gd\, A &\leq & \gd\, A/\ann\, T+\pd _A\,(A/\ann\, T)  \\
   &\leq & \gd\, B + 1+\pd _A\,(A/\ann\, T).
\end{eqnarray*}
\end{proof}

\begin{obs}
We claim that $\mbox{pd}_A\,(A/\mbox{ann}\, T)\leq \mbox{pd}_A\, T \leq \mbox{pd}_A\,(A/\mbox{ann}\, T) + 1$. In fact, since $T$ is a tilting $(A/\mbox{ann}\, T)$-module then ${\mbox{pd}_{A/\mbox{ann}\, T} T\leq 1}$. Thus, there exists an exact sequence:
\begin{equation}\label{rp3}
  0\rightarrow P_1 \rightarrow P_0 \rightarrow T \rightarrow 0
\end{equation}
with $P_0,P_1\in \mbox{add}\, (A/\mbox{ann}\, T)$. Since  $\mbox{pd}_A\,P_i\leq \mbox{pd}_A\,(A/\mbox{ann}\, T)$, for $i=0,1$, by  \cite[Apendix 4, Proposition 4.7]{ASS} we get that $\mbox{pd}_A\, T \leq \mbox{pd}_A\,(A/\mbox{ann}\, T) + 1$.

On the other hand, there exists a short exact sequence:
\begin{equation}\label{eqninclinante}
  0\rightarrow A/\mbox{ann}\, T \rightarrow T_0 \rightarrow T_1 \rightarrow 0
\end{equation}
with $T_0,T_1\in \mbox{add}\, T$. It follows from \cite[Apendix 4, Proposition 4.7]{ASS}, that ${\mbox{pd}_A\,(A/\mbox{ann}\, T)\leq \mbox{pd}_A \,T}$, because ${\mbox{pd}_A\,T_i\leq \mbox{pd}_A\,T}$, for ${i=0,1}$.

As a consequence, there is a relationship between the bound given by Proposition \ref{cotagd1} and the bound given by Y. Miyashita in \cite{Miya} for the global dimension of the endomorphism algebra of a tilting module of finite projective dimension.
\end{obs}

In the following example we show that the bound given in Proposition \ref{cotagd1} for the global dimension of the endomorphism algebra of a $\tau$-tilting module is minimum.

\begin{ej}
Let $A=kQ/I$ be the algebra given by:
$$
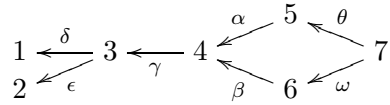
{
    \xymatrix  @!0 @R=0.5cm  @C=1.2cm {
       & &   &5\ar[dl]_{\alpha}&  \\
       1&3\ar[l]_{\delta} \ar[dl]^{\epsilon} &4 \ar[l]^{\gamma}  & & 7 \ar[lu]_{\theta}\ar[ld]^{\omega} \\
      2 & & &6\ar[ul]^{\beta} & }
}
$$
\noindent with $I=<\alpha\gamma,\beta\gamma,\gamma\delta,\theta\alpha-\omega\beta>$. Observe that $\mbox{gldim}\, A =4$.

Consider
$T= 6\oplus {\small\txt{7\\5\,\,\,6}}\oplus 3 \oplus {\small\txt{3\\2}} \oplus {\small\txt{3\\1}} \oplus {\small\txt{7\\5\,\,\,6\\4}} \oplus {\small\txt{7\\6}}$ a $\tau$-tilting $A$-module.
In this case, $\mbox{ann}\, T=<\gamma>$. Then $A/\mbox{ann}\, T$ is given by the following disconnected quiver:
$$
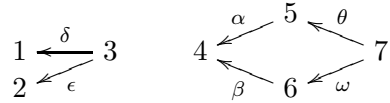
{
    \xymatrix  @!0 @R=0.5cm  @C=1.2cm {
       & &   &5\ar[dl]_{\alpha}&  \\
       1&3\ar[l]_{\delta} \ar[dl]^{\epsilon} &4   & & 7 \ar[lu]_{\theta}\ar[ld]^{\omega} \\
      2 & & &6\ar[ul]^{\beta} & }
}
$$
with $I'=<\theta\alpha-\omega\beta>$.

The module $A/\mbox{ann}\, T \cong 1\oplus 2 \oplus {\small\txt{3\\1\,\,\,2}} \oplus 4 \oplus {\small\txt{5\\ 4}} \oplus {\small\txt{6\\4}} \oplus {\small\txt{7\\5\,\,\,6\\4}} $. Note that all the direct summands of  $A/\mbox{ann}\, T$ are projective $A$-modules except for the direct summand $S_4$. Then, ${\mbox{pd}_A(A/\mbox{ann} T)=\mbox{pd}_A S_4=2}$.

The endomorphism algebra $B=\mbox{End}_AT$ is the following disconnected hereditary algebra:
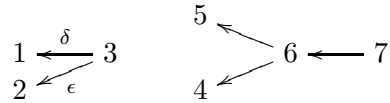
\[
{
    \xymatrix  @!0 @R=0.5cm  @C=1.2cm {
       & &  5 &&  \\
       1&3\ar[l]_{\delta} \ar[dl]^{\epsilon} &   & 6\ar[lu]\ar[ld] & 7 \ar[l] \\
      2 & &4 &  & }
}
\]
Then, $\mbox{gldim}\, B =1$. Therefore
\[4=\mbox{gldim}\, A \leq \mbox{gldim}\, B + \mbox{pd}_A\, (A/\mbox{ann}\, T) + 1 = 4.\]
\end{ej}

Given $A$ an algebra of finite global dimension, it is well-known that the global dimension of the endomorphism algebra of a tilting $A$-module is always finite, see \cite[Chapter VI, Theorem 4.2]{ASS}. More generally, the same fact occurs if we consider the global dimension of the endomorphism algebra of a tilting $A$-module of finite projective dimension, see \cite{Miya}. The global dimension of the endomorphism algebra of a $\tau$-tilting module is not always finite, as we show in the next example.

\begin{ej}\label{infinita}
Consider $A=kQ/I$ the algebra given by the following quiver:
$$
{
    \xymatrix  @!0 @R=0.5cm  @C=1.2cm {
       1\ar@<1ex>[r]^{\alpha}& 2\ar@<1ex>[l]^{\beta}\ar[rd]_{\theta}& &4\ar[ll]_{\lambda}\ar[ld]^{\omega}   \\
       & & 3&
       }
}
$$
\noindent with  $I=<\lambda\theta, \alpha\beta, \lambda\beta\alpha>$.

Observe that $\mbox{gldim}\, A =3$. Consider the $\tau$-tilting $A$-module $T= {\small{\txt{1\\2\\3}}} \oplus \hspace{-0.15in} {\small{\txt{1\,\,\,\,\,\,\,\\2\,\,\,4\\\,\,\,\,\,\,\,3\,\,\,\,2\\ \,\,\,\,\,\,\,\,\,\,\,\,\,\,\,\,\,\,\,1 }}} \hspace{-0.08in}\oplus {\small{\txt{4\\2\,\,\,3\\1\,\,\,\,\,}}} \oplus \hspace{-0.1in} {\small{\txt{\,\,\,\,\,\,\,4\,\,\,1\\3\,\,\,2}}}. $

The annihilator  of $T$ is generated by the path $\beta\alpha$. Then $A/\mbox{ann}\, T$ is given by the following quiver:
$$
{
    \xymatrix  @!0 @R=0.5cm  @C=1.2cm {
       1\ar@<1ex>[r]^{\alpha}& 2\ar@<1ex>[l]^{\beta}\ar[rd]_{\theta}& &4\ar[ll]_{\lambda}\ar[ld]^{\omega}   \\
       & & 3&
       }
}
$$
\noindent with $I'=<\lambda\theta, \alpha\beta, \beta\alpha>$.

The minimal projective resolution of $S_1\in A/\mbox{ann}\, T$ is
$$
{\xymatrix  @!0 @R=0.6cm  @C=1.2cm {
\dots \ar[r] & {\small{\txt{1\\2\\3}}}\ar[rr]\ar[rd]&& {\small{\txt{2\\1\,\,\,3}}}\ar[rr]\ar[rd]&& {\small{\txt{1\\2\\3}}}\ar[r]& S_1 \ar[r] & 0.\\
&& S_1 \ar[ru]&& {\small{\txt{2\\3}}}\ar[ru]&&&&
}}
$$
\noindent Then,  $\mbox{pd}_{A/\mbox{ann}\, T}S_1=\infty$. Thus  $\mbox{gldim}\, (A/\mbox{ann}\, T)=\infty$. Therefore, $\mbox{gldim}\, \mbox{End}_A T=\infty$.
\end{ej}

We end up this section showing a family of algebras where the global dimension of the endomorphism algebra of a $\tau$-tilting module is infinite.

In general, for all $n=k+ 3$, with $k\geq 0$, we found an algebra $A_n$ of global dimension equal to $n$ and, for each  $n$, a $\tau$-tilting $A_n$~-module $T_n$ such that $B_n=\mbox{End}\, T_n$ is of infinite global dimension. In fact, consider ${A_n=kQ_n/I_n}$, where $Q_n$ is as follows
$$
{
    \xymatrix  @!0 @R=0.7cm  @C=1.2cm {
       1\ar@<1ex>[r]^{\alpha}& 2\ar@<1ex>[l]^{\beta}\ar[rd]_{\theta}& &4\ar[ll]_{\lambda}\ar[ld]^{\omega}& a_1\ar[l]_{\mu_1} & a_2\ar[l]_{\mu_2} & \dots & a_k\ar[l]_{\mu_n}  \\
       & & 3&
       }
}
$$
and $I_n=<\lambda\theta, \alpha\beta, \lambda\beta\alpha, \mu_1\lambda, \mu_2\mu_1,\mu_3\mu_2, \dots, \mu_{k}\mu_{k-1}>$, for $k\geq1$.

Consider $U_1=\hspace{-0.17in}{\small{\txt{1\,\,\,\,\,\,\,\\2\,\,\,4\\\,\,\,\,\,\,\,3\,\,\,\,2\\ \,\,\,\,\,\,\,\,\,\,\,\,\,\,\,\,\,\,\,1 }}}$. The minimal projective presentation of  $U_1$ (in any algebra $A_n$) is
\[ P_3\rightarrow P_1\oplus P_{4} \rightarrow U_1\rightarrow 0\]
Then, we have the following exact sequence
\[0\rightarrow \tau U_1 \rightarrow I_3 \rightarrow I_1\oplus I_{4}.\]
Hence, $\tau U_1 \cong {\small{\txt{2\\3}}}$. Similarly, if we consider $U_2=\hspace{-0.15in}{\small{\txt{\,\,\,\,\,\,\,4\,\,\,1\\3\,\,\,2}}}$. The minimal projective presentation of $U_2$ is as follows
\[P_2\rightarrow P_1\oplus P_{4} \rightarrow U_2\rightarrow 0\]
Thus, we have
\[0\rightarrow \tau U_2 \rightarrow I_2 \rightarrow I_1\oplus I_{4}.\]
Hence, $\tau U_2 \cong S_2$.

Consider the module $T_n=\bigoplus\limits_{i=a_1}^{a_k}P_i\oplus U_1\oplus U_2 \oplus P_1\oplus P_3 $. Since $\tau_{A_n} T_n \cong {\small{\txt{2\\3}}} \oplus 2$, then we get that $T_n$ is a $\tau$-tilting $A_n$-module.

It is not hard to see that $\mbox{ann}\, T_n =<\beta\alpha>$. Furthermore, $\mbox{gldim}\, A/\mbox{ann}\, T=\infty$. Therefore, $\mbox{gldim}\, \mbox{End}\, T_n= \infty$.

\section{Algebras of global dimension two}
The aim of this section is to  study the global dimension of the endomorphism algebra of a $\tau$ tilting module over an  algebra of global dimension two. We focus our attention in two families of algebras; the monomials algebras and the special biserial algebras.  We prove that if $T$ is a $\tau$-tilting module over these algebras then its endomorphism algebra is of finite global dimension. We conjecture that this result holds in general for any algebra of global dimension two.

We start studying the generators of the annihilator of a $\tau$-tilting module. We show that the generators of the annihilator of a $\tau$-tilting module are related with the relations in the algebra $A$. This is the main key for our purposes.

\begin{teo}\label{anulador}
Let $A=kQ/I$ be an algebra, $T$ be a $\tau$-tilting $A$-module and ${\rho \in \emph{ann}\, T}$ be  a non-zero path of length greater than or equal to two such that no proper subpath of $\rho$ belongs to $\emph{ann}\,T$. Then one of the following conditions hold:
\begin{enumerate}[i)]
\item[i)] There exists a non-zero path $\gamma$ such that $\gamma\rho $ is a zero-relation in $A$, or
\item[ii)] There exist non-zero paths $\gamma_i, \gamma$ and non-zero scalars  $\lambda_i\in k$ such that ${\gamma\rho + \sum\limits_{i\in I}\lambda_i\gamma_i}$ is a minimal relation in $A$, where $I$ is a finite set of indices.
\end{enumerate}
\end{teo}
\begin{proof}
Assume that conditions $i)$ and $ii)$ do not hold.

Let $T=\bigoplus\limits_{i=1}^{r}T_i$ be a $\tau$-tilting $A$-module, with each $T_i$ an indecomposable $A$-module for ${i=1,\dots,r}$. Let $\rho=\alpha_1\dots \alpha_n$ be a path in $Q$ from $a$ to $c$, with $\alpha_j\in Q_1$ for $j=1,\dots,n$, such that $\rho\in \ann\,T$  and where no proper subpath of $\rho$ belongs to $\ann\,T$. In particular, ${\alpha_1\ldots\alpha_{n-1}\notin \ann\, T}$ and ${\alpha_2\ldots \alpha_n\notin \ann\, T}$. Since ${\ann\, T =\bigcap\limits_{j=1}^r \ann\, T_j}$, then there exist ${T_l=((T_l)_i,\varphi_{\lambda})_{i\in Q_0,\,\lambda\in Q_1}}$ and ${T_h=((T_h)_i,\vartheta_{\lambda})_{i\in Q_0,\,\lambda\in Q_1}}$ indecomposable direct summands of $T$ such that ${\alpha_1\ldots\alpha_{n-1}\notin \ann\, T_l}$ and ${\alpha_2\ldots \alpha_n\notin \ann\, T_h}$.

Let ${d_i=\mbox{dim}_k(T_l)_i}$ and $P(i)$ be the indecomposable projective $A$-module corresponding to the vertex $i$. Consider ${P_0=\bigoplus\limits_{i\in Q_0}d_iP(i)}$, where $d_iP(i)$ is the direct sum of  $d_i$ copies of $P(i)$. To define a morphism $g:P_0\rightarrow T_l$ we introduce specific bases for each of the representations $T_l$ and $P_0$. For each $i\in Q_0$, let  $\{m_{i_1},\ldots,m_{i_{d_i}}\}$ be a basis for $(T_l)_i$, and thus
$B'=\{m_{ij}\,| \,i\in Q_0,\, j=1,\ldots,d_i\}$
is a basis for $T_l$. For the projective modules we consider the basis defined in \cite[Chapter III, Lemma 2.1]{ASS}. Let
\[B=\{c_{ij}\,|\,i\in Q_0,\, c_i \,\,\, \textrm{a path with}\,\,\, s(c_i)=i,\,j=1,\ldots,d_i \}\]
be a basis of $P_0$. We define $g:P_0\rightarrow T_l$ on the basis $B$ as follows:
 \[g(c_{ij})=\varphi_{c_i}(m_{ij})\in (T_l)_{t(c_i)}\]
 and we extend $g$ linearly to $P_0$.
It follows that  $g$ is an epimorphism, but in general it is not minimal. If we denote by $(\widetilde{P_0},\widetilde{g})$ the projective cover of $T_l$, there exists a commutative diagram:
$$
{
    \xymatrix  @!0 @R=1.5cm  @C=2cm {
      P_0\ar[r]^{g}\ar[d]^{\pi}&T_l\ar[r]\ar[d]& 0 \\
      \widetilde{P_0}\ar[r]^{\widetilde{g}}&T_l\ar[r]&0
     } }
$$
such that $g=\widetilde{g}\pi$ y $\widetilde{P_0} $ is a direct summand of $P_0$.

Since $\alpha_1\ldots\alpha_{n-1}\notin \ann\, T_l$, we have that $\varphi_{\alpha_{n-1}}\dots \varphi_{\alpha_1}\neq 0$. Moreover, there exists a non-zero path  $\epsilon_1\dots\epsilon_m$ from $x$ to $a$ in $Q$ such that $S(x)$ is a direct summand of $\mbox{top}\,T_l$ and $\varphi_{\alpha_{n-1}}\dots \varphi_{\alpha_1}\varphi_{\epsilon_m}\dots\varphi_{\epsilon_1}\neq 0$. Thus, there exists a non-zero element $m_{xr_x}$ of the basis of $(T_l)_x$ such that $\varphi_{\rho}\varphi_{\epsilon_m}\dots\varphi_{\epsilon_1}(m_{xr_x})=m\neq 0$.

Since $P(x)$ is a direct summand of $\widetilde{P_0}$, $g$ and $\widetilde{g}$ coincide in $\widetilde{P_0}$, then
\[\widetilde{g}(e_x)=g(c_{xr_x})=\varphi_{e_x}(m_{xr_x})=m_{xr_x}.\]
Then, $\widetilde{g}(\epsilon_1\dots \epsilon_m\alpha_1\dots\alpha_{n-1})=\varphi_{\epsilon_1\dots \epsilon_m\alpha_1\dots\alpha_{n-1}}(m_{xr_x})=m\neq 0$. Moreover, since $\alpha_1\ldots\alpha_{n-1}\alpha_n$ does not satisfy $i)$, then the path $\epsilon_1\dots \epsilon_m\alpha_1\dots\alpha_{n}$ is non-zero and $\widetilde{g}(\epsilon_1\dots \epsilon_m\alpha_1\dots\alpha_{n})=\varphi_{\alpha_n}\ldots \varphi_{\alpha_1}\varphi_{\epsilon_m}\dots\varphi_{\epsilon_1}(m_{xr_x})=0$, because $\varphi_{\alpha_n}\dots \varphi_{\alpha_1}=0$. Therefore, $0\neq \epsilon_1\dots \epsilon_m\alpha_1\dots\alpha_{n} \in \mbox{ker}\,\widetilde{g}$.

We claim that $\epsilon_1\dots \epsilon_m\alpha_1\dots\alpha_{n}\notin \mbox{rad}(\mbox{ker}\,\widetilde{g})$ and, in consequence $S(c)$ is a direct summand of $\mbox{top}\,(\mbox{ker}\,\widetilde{g})$. Indeed, since $\epsilon_1\dots \epsilon_m\alpha_1\dots\alpha_{n}\in (\mbox{ker}\,\widetilde{g})_c$, according to \cite[Chapter III, Lemma 2.2]{ASS} we have that
\begin{equation*}
\epsilon_1\dots \epsilon_m\alpha_1\dots\alpha_{n}\in \sum\limits_{\lambda:\bullet\rightarrow c }\mbox{Im}(\psi_{\lambda}:(\mbox{ker}\,\widetilde{g})_{\bullet}\rightarrow (\mbox{ker}\,\widetilde{g})_c).
\end{equation*}
In the representation of $\mbox{ker}\,\widetilde{g}$, the morphisms $\psi_{\lambda}$ are the restriction of the morphisms in  $\widetilde{P_0}$, i.e, the multiplication by arrows. If  $\epsilon_1\dots \epsilon_m\alpha_1\dots\alpha_{n}\in \mbox{rad}\,(\mbox{ker}\,\widetilde{g})$, then
\begin{equation*}
\epsilon_1\dots \epsilon_m\alpha_1\dots\alpha_{n}= \sum\limits_{\lambda:y\rightarrow c }\lambda\eta_y.
\end{equation*}
with $\eta_y\in (\mbox{ker}\widetilde{g})_y$, getting  a contradiction, because condition $ii)$ does not hold. Therefore, $S(c)$ is a direct summand of $\mbox{top}\,(\mbox{ker}\,\widetilde{g})$.

If $P_1\stackrel{f}\rightarrow \widetilde{P_0}\stackrel{\widetilde{g}}\rightarrow T_l\rightarrow 0$ is the minimal projective presentation of $T_l$, then $P(c)$ is a direct summand of $P_1$. Moreover, by construction,  ${f(e_c)=\epsilon_1\dots \epsilon_m\alpha_1\dots\alpha_{n}}$.

On the other hand, since $\alpha_2\ldots \alpha_n\notin \ann\, T_h$, then $\vartheta_{\alpha_n}\ldots\vartheta_{\alpha_2}\neq 0$. Thus, $S(c)$ is a composition factor of $T_h$. Then, there exists a non-zero element $t\in (T_h)_c$. We define $\theta:P_c\rightarrow T_h$ such that $\theta(e_c)=t\neq0$ and if $w$ is a path starting at the vertex $c$, $\theta(w)=\vartheta_w(t)$. Consider the morphism ${\widetilde{\theta}:P_1=P(c)\oplus P_1'\rightarrow T_h}$ defined by $\widetilde{\theta}=(\theta, 0)$.

Since $\Hom_A(T_h,\tau_AT_l)=0$, by Proposition \ref{resproytau} the morphism \[\Hom_A(\widetilde{P_0},T_h)\stackrel{(f,Y)}{\rightarrow}\Hom_A(P_1,T_h)\]
is surjective. Then, there exists a morphism $\psi:\widetilde{P_0}\rightarrow T_h$ such that $\widetilde{\theta}=\psi f$. In particular,  $t=\widetilde{\theta}(e_c,0)=\psi f (e_c)=\psi (\epsilon_1\dots \epsilon_m\alpha_1\dots\alpha_{n})$.

We have the following commutative diagram:
$$
{
    \xymatrix  @!0 @R=1.5cm  @C=2cm {
      (\widetilde{P_0})_a \ar[r]^{*\alpha_1}\ar[d]^{\psi_a}&\dots\dots\ar[r]^{*\alpha_n}&(\widetilde{P_0})_c \ar[d]^{\psi_c} \\
      (T_h)_a\ar[r]^{\varphi_{\alpha_1}}&\dots\dots\ar[r]^{\varphi_{\alpha_n}}&(T_h)_c
     } }
$$
It follows that $\psi_c *\alpha_n\dots *\alpha_1=\varphi_{\alpha_n}\ldots \varphi_{\alpha_1}\psi_a$, where $*\alpha_i$ is the morphism of the representation of $P_0$ which is the multiplication by $\alpha_i$. Then,
\[\psi_c *\alpha_n\dots*\alpha_1(\epsilon_1\dots  \epsilon_m)=\psi_c(\epsilon_1\dots\epsilon_m\alpha_1\dots\alpha_n)=t\]
\[\varphi_{\alpha_n}\dots \varphi_{\alpha_1}\psi_a(\epsilon_1\dots\epsilon_m)=0.\]
Hence $t=0$, which is a contradiction.
\end{proof}

Next, we show an example of certain $\tau$-tilting module that satisfies statements $i)$ and $ii)$ of the above proposition.

\begin{ej}
Let $A=kQ/I$ be the algebra given by the following quiver:
$$
{
    \xymatrix  @!0 @R=0.8cm  @C=0.9cm {
      &2\ar[r]^{\gamma}&4\ar[rd]^{\lambda}&&&\\
      1\ar[rd]^{\beta}\ar[ru]^{\alpha}&&&5\ar[r]^{\epsilon}&6\ar[r]^{\mu}&7\\
      &3\ar[rru]^{\delta}&&&&
     } }
$$
where $I=<\alpha\gamma\lambda-\beta\delta,\,\lambda\epsilon\mu>$.

Consider the $\tau$-tilting $A$-module $$T=\txt{1\\2\\4}\oplus \txt{4\\5\\6}\oplus \txt{6\\7}\oplus \txt{1\\3}\oplus 7\oplus 4 \oplus 1.$$
If we compute the annihilator of $T$, we get that $\mbox{ann}\,T=<\delta, \gamma\lambda, \epsilon\mu>$. Thus, the path $\gamma\lambda$ is associated to the minimal relation $\alpha{\gamma\lambda}-\beta\delta$, and $\epsilon\mu$ is related with the zero-relation $\lambda{\epsilon\mu}$.
\end{ej}

In the next result we give a nice presentation of $A/\ann\,T$, when $T$ is a $\tau$-tilting module.

\begin{lema}\label{presentacion}
Let $A$ be an algebra and $\eta_A:kQ_A\rightarrow A$ be a presentation of $A$, with ${I_A=\emph{ker}\,\eta_A}$. Let $T$ be a $\tau$-tilting module such that $\emph{ann}\,T$ is generated by paths. Then there exits a presentation ${\eta_{A/\emph{ann}\,T}:kQ_{A/\emph{ann}\,T}\rightarrow A/\emph{ann}\,T}$ such that $\eta_{A/\emph{ann}\,T}\widetilde{\pi}=\pi\eta_A$, where ${\pi:A\rightarrow A/\emph{ann}\,T}$ is the projection morphism and ${\widetilde{\pi}:kQ_A\rightarrow kQ_{A/\emph{ann}\,T}}$ is a surjective map. Moreover, $$I_{A/\emph{ann\,T}}=\emph{ker}\,\eta_{A/\emph{ann}\,T}=\langle I_A\cap kQ_{A/\emph{ann}\,T},\,\, \emph{ann}\,T\cap kQ_{A/\emph{ann}\,T}\rangle.$$
\end{lema}
\begin{proof}
Let $Q$ be the quiver define as follows:
\begin{eqnarray*}
  Q_0 &=& (Q_A)_0. \\
  Q_1 &=& (Q_A)_1\setminus \{\alpha\in (Q_A)_1 \mid \, \alpha\in \ann\,T\}.
\end{eqnarray*}
Then we have a surjective map $\widetilde{\pi}:kQ_A\rightarrow kQ$ as follows:
\begin{eqnarray*}
  \widetilde{\pi}(e_x) &=& e_x \,\,\mbox{for all}\,\, x\in (Q_A)_0. \\
  \widetilde{\pi}(\alpha) &=& 0 \,\,\mbox{if}\,\, \alpha\in \ann\,T\cap (Q_A)_1.\\
   \widetilde{\pi}(\beta) &=& \beta \,\, \mbox{if} \,\, \beta \in (Q_A)_1 \,\, \mbox{and} \,\, \beta\notin \ann\,T.
\end{eqnarray*}
It is clear that ${\mbox{ker}\,\widetilde{\pi}=\langle\alpha\in (Q_A)_1 \,\mid\,\alpha \in \ann\,T\rangle}$. Consider ${\pi:A\rightarrow A/\ann\,T}$ the usual projection. For $\alpha\in \ann\,T$ we have that ${\pi\eta_A(\alpha)=0}$. Thus, ${\mbox{ker}\,\widetilde{\pi}\subset \mbox{ker}\,\pi\eta_A}$. Then, there exists a unique morphism ${\eta_{A/\ann\,T}:kQ\rightarrow A/\ann\,T}$ such that ${\eta_B\widetilde{\pi}=\pi\eta_A}$, as we can see in the following diagram:
$$\xymatrix @1 @R=0.4cm  @C=0.8cm{
kQ_A\ar[dd]_{\widetilde{\pi}} \ar[rr]^{\pi\eta_A}&&A/\ann\,T\\
&&\\
kQ\ar@{-->}[uurr]_{\eta_{A/\tiny{\mbox{ann}}\,T}}
}$$
We claim that $I_{A/\tiny{\ann}\,T}=\mbox{ker}\,\eta_{A/\tiny{\ann}\,T}$ is an admissible ideal. 

First, we prove that $I_{A/\tiny{\ann}\,T} \subset r_Q^2$. Assume that $I_{A/\tiny{\ann}\,T}$ is not included in $r_Q^2$. Consider $\gamma \in I_{A/\tiny{\ann}\,T}$ and $\gamma\notin  r_Q^2$. Then, there exist arrows $\beta_1,\dots, \beta_m$  in $Q$, $c_1,\dots,c_m$ non-zero scalars and $\rho\in r_Q^2$ such that
$$\gamma=\sum\limits_{i=1}^mc_i\beta_i + \rho. $$
Consider $\gamma$  an element of $kQ_A$. We get that
$$\pi \eta_A (\gamma)=\eta_{A/\tiny{\ann}\,T}\widetilde{\pi}(\gamma)=\eta_{A/\tiny{\ann}\,T}(\gamma)=0.$$
Hence, $\eta_A(\gamma)= \gamma + I_A \in \mbox{ker}\,\pi = \ann \,T$. Then there exist $a_1,\dots, a_l$ non-zero scalars, $\alpha_1,\dots, \alpha_l$ arrows that belong to $\ann\,T$ and $\lambda\in \ann\,T \cap r_{Q_A}^2$ such that
$$\gamma + I_A = \sum\limits_{j=1}^{l}a_j\alpha_j + \lambda + I_A.$$
Since $I_A$ is an admissible ideal and $\rho, \lambda \in r_{Q_A}^2$ the equation
$$\sum\limits_{i=1}^mc_i\beta_i + \rho+I_A = \sum\limits_{j=1}^{l}a_j\alpha_j + \lambda + I_A$$
implies that $\sum\limits_{i=1}^mc_i\beta_i=\sum\limits_{j=1}^{l}a_j\alpha_j$, which is a contradiction because the arrows $\alpha_j$ are not in $Q$, proving that $I_{A/\tiny{\ann}\,T} \subset r_Q^2$.

On the other hand, there exists $n\in \mathds{N}$ such that $r_{Q_A}^n \subset I_A$. Since $Q$ is a subquiver of $Q_A$ we have that $r_{Q}^n\subset r_{Q_A}^n$. Thus, $r_{Q}^n \subset I_A$. Let $x\in r_{Q}^n $, then
$$\eta_{A/\tiny{\ann}\,T}(x)=\eta_{A/\tiny{\ann}\,T}\widetilde{\pi}(x)= \pi\eta_A(x)=0$$
because $x\in I_A$. Hence, $r_Q^n \subset I_{A/\tiny{\ann}\,T}$ and $I_{A/\tiny{\ann}\,T}$ is an admissible ideal.

Since the quiver of an algebra is uniquely determined, then ${Q=Q_{A/\tiny{\ann}\,T}}$ and ${\eta_{A/\tiny{\ann}\,T}:kQ\rightarrow A/\ann\,T}$ is a presentation of $A/\ann\,T$.

Next, we prove that $I_{A/\tiny{\ann}\,T}=\mbox{ker}\eta_{A/\tiny{\ann}\,T}= \langle I_A\cap kQ_{A/\tiny{\ann}\,T},\,\, \mbox{ann}\,T\cap kQ_{A/\tiny{\ann}\,T}\rangle$. Let ${y\in I_{A/\tiny{\ann}\,T}}$. Then,
$$0=\eta_{A/\tiny{\ann}\,T}(y)=\eta_{A/\tiny{\ann}\,T}(\widetilde{\pi}(y))=\pi\eta_A(y).$$
Thus, $\eta_A(y)\in \mbox{ker}\,\pi = \ann\,T$. Then there exists ${z\in \ann\,T}$ such that ${\eta_A(y)=y+I_A=z}$. Hence, ${y\in <I_A\cap kQ_{A/\tiny{\ann}\,T},\,\, \mbox{ann}\,T\cap kQ_{A/\tiny{\ann}\,T}>}$.

Conversely, let $z\in <I_A\cap kQ_{A/\tiny{\mbox{ann}}\,T},\,\, \mbox{ann}\,T\cap kQ_{A/\tiny{\mbox{ann}}\,T}>$. Then $z=z_1+z_2$, with $z_1\in I_A\cap kQ_{A/\tiny{\ann}\,T}$ and $z_2\in \ann\,T \cap kQ_{A/\tiny{\ann}\,T}$. Thus
\begin{eqnarray*}
  \eta_{A/\tiny{\ann}\,T}(z) &= & \eta_{A/\tiny{\ann}\,T}(\widetilde{\pi}(z)) \\
   &=& \pi \eta_A (z) \\
   &=& \pi (\eta_A(z_1))+\pi( \eta_A (z_2))\\
   &=& 0.
\end{eqnarray*}
Therefore, $z\in I_{A/\tiny{\ann}\,T}$.
\end{proof}

\subsection{Monomial algebras of global dimension two}

We start  proving that the annihilator of a $\tau$-tilting module over a monomial algebra is an ideal generated by paths.

 \begin{lema}\label{anuladormonomial}
Let $A=kQ/I$ be a monomial algebra and $T$ be a $\tau$-tilting  $A$-module. Suppose that $\sum\limits_{i=1}^m\lambda_i\rho_i\in \emph{ann}\, T$, with $\rho_i$ non-zero paths of length  greater than or equal to two and $\lambda_i\in k$ for $i=1,\dots,m$. Then, $\rho_i\in\emph{ann}\, T$ for $i=1,\ldots, m$.
 \end{lema}

\begin{proof}
Let $T=\bigoplus\limits_{i=1}^nT_i$ be a $\tau$-tilting $A$-module, with each $T_i$ an indecomposable $A$-module, for $i=1,\dots, n$. Consider $\rho=\sum\limits_{i=1}^m\lambda_i\rho_i$ in $Q$,  with $\rho_i$ non-zero paths from $a$ to $c$ of length greater than or equal to two and $\lambda_i\in k$ for $i=1,\dots,m$, such that $\rho\in \ann\,T$.  Suppose that there exists $k\in\{1,\ldots,m\}$ such that ${\rho_k\notin \ann\, T}$. Then, there exists an indecomposable direct summand $T_h=((T_h)_y,\varphi_\alpha)_{(y\in Q_0,\,\alpha\in Q_1)}$  of $T$ such that $\rho_k\notin \ann\, T_h$ and, therefore $\varphi_{\rho_k}\neq 0$.

Let $S(x)$ be a direct summand of $\mbox{top}\, T_h$ such that there exists a non-zero path $\gamma$ from $x$ to $a$ and $\varphi_{\rho_k} \varphi_{\gamma}\neq 0$. Let $\{m_{x_1},\ldots, m_{x_{d_x}}\}$ be a basis for $(T_h)_x$ as a  vector space. As in the proof of Proposition \ref{anulador}, we can consider a minimal projective presentation of $T_h$ as follows:
$$P_1\stackrel{g}\rightarrow P_0\stackrel{f}\rightarrow T_h\rightarrow0$$
such that $f(e_x)=m_{x_1}$ and $f(\epsilon)=\varphi_{\epsilon}(m_{x_1})$ if $\epsilon$ is a path in $Q$ starting in $x$. Since $A$ is a monomial algebra and $\gamma\rho_k\neq 0$ then $\gamma \sum\limits_{i=1}^m\lambda_i\rho_i\neq 0$. Therefore, $\gamma \sum\limits_{i=1}^m\lambda_i\rho_i \in \mbox{ker}f$, because $f(\gamma \rho)=\varphi_{\rho }\varphi_{\gamma}=0\varphi_{\gamma}=0$. Next, we show that $\gamma \rho \notin \mbox{rad}(\mbox{ker}f)$.

We recall that the morphisms in the representation of  the $\mbox{ker}f$ are the restriction of the morphism in the representation of $P_0$, and therefore they consist in the multiplication by the correspondent arrows. Then, if $\gamma \rho \in \mbox{rad}(\mbox{ker}f)$,  there exists $\omega_y\in (\mbox{ker}\,f)_c$ such that
\begin{equation}\label{eqnrelacion}
\gamma \rho=\sum\limits_{\beta:y\rightarrow c}\omega_y\beta.
\end{equation}

The equation (\ref{eqnrelacion}) define a relation $\gamma \rho-\sum\limits_{\beta:y\rightarrow c}\omega_y\beta$ in $I$ with  $\gamma\rho\neq0$ which is not monomial, getting a contradiction because  $A$ is a monomial algebra. Therefore, $\gamma \rho \notin \mbox{rad}\,(\mbox{ker}\,f)$. Hence, $S(c)$ is a direct summand of $\mbox{top}\,T_h$ and $P(c)$ is a direct summand of $P_1$.

As  in the proof of Proposition \ref{anulador}, if $P(c)$ is a direct summand of  $P_1$ we can construct a morphism $\theta:P_1\rightarrow T_h$ which does not factorize through $g$,
 a contradiction to Proposition \ref{resproytau}. Hence, if $\sum\limits_{i=1}^m\lambda_i\rho_i\in \ann\, T$, then $\rho_i\in \ann\, T$ for all $i=1,\dots,m$.
\end{proof}

\begin{obs}\label{monomial2}
Observe that if  $A=kQ/I$ is a monomial algebra of global dimension two, with $I=<\rho>$ there is no path of the form $\gamma_1\gamma_2\gamma_3$, with $\gamma_1,\gamma_2,\gamma_3$  non-trivial paths, such that $\gamma_1\gamma_2$ and $\gamma_2\gamma_3$ are the only zero-relations in $A$ contained in the path. Indeed, assume that there exists such a path in $A$. Denote $x=s(\gamma_1)$ and $(\widetilde{Q},\widetilde{\rho})$ the universal cover of $(Q,\rho)$ with $\Phi: \widetilde{Q}\rightarrow Q$. Then there exists in $\widetilde{Q}$ two relations $r_1^*, r_2^*$ associated to $x^*$, with $\Phi(x^*)=x$, such that $\Phi(r_1^*)=\gamma_1\gamma_2$, $\Phi(r_2^*)=\gamma_2\gamma_3$ and $s(r_1^*)<s(r_2^*)<e(r_1^*)$. Then by Theorem \ref{arbolresproy}, we have that $\mbox{pd}_{\tiny{\mbox{rep}(\widetilde{Q},\widetilde{\rho})}}S(x^*)\geq 3$ . Thus, $\mbox{pd}_{A}S(x)\geq 3$ which is a contradiction.
\end{obs}

Now we are in position to prove Theorem C.

\begin{teo}
Let $A=kQ/I$ be a monomial algebra of global dimension two, $T$ be a $\tau$-tilting $A$-module and $B=\emph{End}_A\,T$. Then $\emph{gldim}\, B < \infty$.
\end{teo}
\begin{proof}
Assume that $\gd\, B=\infty$. Since $T$ is a tilting $A/\ann\, T$-module, then $\gd\, A/\ann\, T=\infty$.

By Proposition \ref{anulador} and by Lemma \ref{anuladormonomial} we know that $\ann\, T$ is generated by paths. Thus,  $A/\ann\, T$ is a monomial algebra. Since $\gd\, A/\ann\, T = \infty$, there exists a vertex ${v_0\in Q_0}$ such that $\pd_{A/\ann\, T}\,S(v_0)=\infty$. By Proposition \ref{propciclo} and Remark \ref{obsciclo} there exists a sequence of relations associated to $v_0$, $r_1,\dots,r_k$ which ends in a cycle $\rho=\alpha_1\dots\alpha_n$, with $\alpha_i\in (Q_{A/\tiny{\ann}\,T})_1$ for $i=1,\dots,n$ and ${s(\alpha_1)=e(\alpha_n)}$. We denote by ${r'_{1},r'_{2},\dots,r'_{h}}$ the subsequence of relations that go through the cycle, with  $r'_j=\gamma_j\gamma_{j+1}$, $\gamma_j$ non zero paths for $j=1,\dots, h$ and $r_h=\gamma_h\widetilde{\gamma}_1$, with $\widetilde{\gamma}_{1}$ a subpath of $\gamma_1$.

Given $r'_{j}$ in this subsequence, then either $r'_{j}$ is a zero-relation of  $A$ or is a relation of $A/\ann\,T$ which is not a relation in  $A$. We claim that none relation $r'_{j}$  belong to $A$. In fact, assume that there exists $j$ such that $r'_{j}$ is a zero-relation in $A$. Then if $r'_{j-1}$ is a zero-relation in $A$, by Remark \ref{monomial2}  we get that $\mbox{pd}_A\,S(x_{j-1})\geq3$, where ${x_{j-1}=s(r_{j-1})}$, which is a contradiction. Thus, $r_{j-1}$ is not a relation in $A$. Then, by Proposition \ref{anulador} there exists $\gamma$ a non-zero path in $A$ such that $\gamma r'_{j-1}$ is a monomial relation in $A$. Thus, $\mbox{pd}_AS(x)\geq3$, with $x=s(\gamma)$, which is a contradiction. Hence, none $r'_{j}$ belong to $A$.

Since $A$ is a finite dimensional algebra and there is a cycle in $Q$, it must exist a monomial relation  $\lambda \in I$ which starts and ends in the vertices of the cycle. Since the sequence of relations $r'_{1},\dots,r'_{h}$ go through the cycle $\rho$, there exist $l\in\{ 1,\dots, h\}$ and $\nu$ non-zero paths such that $r'_l=r''_1\nu$ and $\lambda=\nu\lambda'$. Let $\gamma'$ be a non-zero path in $A$ such that $\gamma' r_l$ is a zero-relation in $A$. Then,  $\mbox{pd}_AS(y)\geq3$ with $y=s(\gamma')$ which is a contradiction.  Therefore, $\gd\, A/\ann\, T<\infty$ and $\gd\, B<\infty$.
\end{proof}

In the above theorem we show that if  $A$  is a monomial algebra with global dimension  two, then $\gd\,B<\infty$, where $B=\End_A\,T$ and $T$ is a $\tau$-tilting $A$-module. A natural question is if there exists an explicit bound for $\gd\,B$.

In the next example, we show that the above question is not true. More precisely, if we fix a bound $l$ then we can construct an algebra $A$ and a $\tau$-tilting module $T$ such that $\gd\,\mbox{End}_AT=n$, with $n>l$.

\begin{ej}\label{ejemplomonomialsincota}
Consider the algebra $A_n=kQ_n/I_n$, where, for each $n$, $Q_n$ is the following quiver:
$$
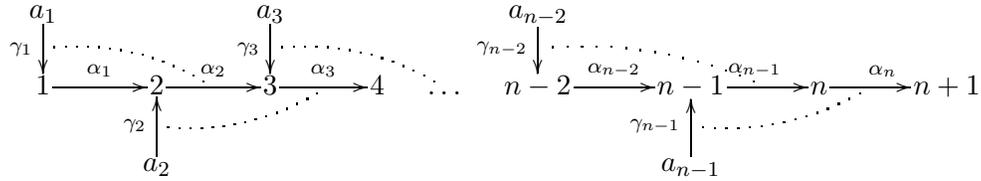
\xymatrix @1 @R=0.3cm  @C=0.5cm{
a_1 \ar[dd]_{\gamma_1}&&&& a_3 \ar[dd]_{\gamma_3}&&&&a_{n-2}\ar[dd]_{\gamma_{n-2}}&&&&&&\\
\ar@{.}@/^/[drrr]&&&&\ar@{.}@/^/[drrr]&&&& \ar@{.}@/^/[drrr]  &&&&&&\\
1\ar[rr]^{\alpha_1}&&2\ar[rr]^{\alpha_2}&&3\ar[rr]^{\alpha_3}&&4&\dots&n-2\ar[rr]^{\alpha_{n-2}}&&n-1\ar[rr]^{\alpha_{n-1}}&&n\ar[rr]^{\alpha_n}&&n+1\\
&&\ar@{.}@/_/[urrr]&&&&&&&&\ar@{.}@/_/[urrr]&&&&\\
&&a_2\ar[uu]^{\gamma_2}&&&&&&&&a_{n-1}\ar[uu]^{\gamma_{n-1}}&&&&
}$$
and  $I_n=<\gamma_i\alpha_i\alpha_{i+1}\,\mid\, i=1,\dots,n-1>$. Since $A$ is a triangular monomial algebra, by \cite{G}, we know that $\gd\,A=2$. Consider the following module
\[T=\bigoplus \limits_{j=1}^{n-1}T_{a_j}\oplus S(a_j) \oplus P(n)\oplus P(n+1)\]
where $S(a_j)$ is the simple at the vertex $a_j$, $P(n)$ and $P(n+1)$ are the projective modules corresponding to the vertices $n$ and $n+1$, respectively, and ${T_{a_j}=((T_{a_j})_x, \varphi^{a_j}_{\delta})_{\{x\in Q_0,\, \delta\in Q_1\}}}$, where for each $x\in (Q_n)_0$, the vector space $(T_{a_j})_x$ is:
$$
(T_{a_j})_x=\left\{
  \begin{array}{ll}
    k & \hbox{if}\,\, x\in\{a_j,a_{j+1},j,j+1\}. \\
    0 & \hbox{else}.
  \end{array}
\right.
$$
and, for each $\delta \in (Q_n)_1$, the morphism $\varphi^{a_j}_{\delta}$ is:
$$
 \varphi^{a_j}_{\delta}=\left\{
  \begin{array}{ll}
    \mbox{Id} & \hbox{if}\,\, \delta\in\{\gamma_j,\gamma_{j+1},\alpha_j\}. \\
    0 & \hbox{else}.
  \end{array}
\right.
$$
Let $P(j)\rightarrow P(a_j)\rightarrow S(a_j)\rightarrow 0$ be the minimal projective presentation of $S(a_j)$. Then we get the following exact sequence:
\[0\rightarrow \tau_AS(a_j)\rightarrow I(j) \rightarrow I(a_j)\cong S(a_j).\]
Thus, we have that $\tau_AS(a_j)=((\tau_AS(a_j))_x,\phi_{\delta}^{a_j})$ is given by
$$
(\tau_AS(a_j))_x=\left\{
  \begin{array}{ll}
    k & \hbox{if}\,\, x\in\{a_{j-1}1,\dots,j\} \\
    0 & \hbox{else}
  \end{array}
\right.$$
for each $x\in (Q_n)_0$ and the morphism is
$$ \phi^{a_j}_{\delta}=\left\{
  \begin{array}{ll}
    \mbox{Id} & \hbox{if}\,\, \delta\in\{\gamma_{j-1},\alpha_1,\dots,\alpha_{j-1}\} \\
    0 & \hbox{else}
  \end{array}
\right.
$$
for each $\delta\in (Q_n)_1$.

Since $\mbox{top}\,T_{a_j}=S(a_{j+1})\oplus S(a_j)$, then the minimal projective presentation of  $T_{a_j}$ ${P(j+1)\rightarrow P(a_{j+1})\oplus P(a_j)\rightarrow T(a_j)\rightarrow 0}$ induces an exact sequence
\[0\rightarrow \tau_AT_{a_j}\rightarrow I(j+1)\rightarrow S(a_{j+1})\oplus S(a_j)\]
Hence, we get that $\tau_AT_{a_j}=((\tau_AT_{a_j})_x,\psi_{\delta}^{a_j})$ is given by
$$
(\tau_AS(a_j))_x=\left\{
  \begin{array}{ll}
    k & \hbox{if}\,\, x\in\{1,\dots,j\}. \\
    0 & \hbox{else}.
  \end{array}
\right.$$
for each $x\in (Q_n)_0$ and the morphism is
$$\psi^{a_j}_{\delta}=\left\{
  \begin{array}{ll}
    \mbox{Id} & \hbox{if}\,\, \delta\in\{\alpha_1,\dots,\alpha_{j-1}\}. \\
    0 & \hbox{else}.
  \end{array}
\right.
$$
for each  $\delta\in (Q_n)_1$.

It follows that $\Hom _A(T,\tau_AT)=0$ . Thus, since  $|T|=n-1+n-1+2=n$ we obtain that $T$ is a $\tau$-tilting module. Finally we compute $\End_A(T)$.

Observe that $\mbox{dim}_k\Hom_A(T_{a_j},T_{a_{j-1}})=1$ and $\mbox{dim}_k\Hom_A(T_{a_j},T_{a_{j-2}})=0$. On the other hand, there is only one linear independent morphism  from $T(a_j)$ to $S(a_j)$, let say $f_j$, and there is only one linear independent morphism  from $T(a_j)$ to $S(a_{j-1})$ which factors through $f_j$. Furthermore, $\mbox{dim}_k\Hom_A(P_n,T_{n-1})=1$, $\mbox{dim}_k\Hom_A(P_{n+1},P_{n})=1$ and $\mbox{dim}_k\Hom_A(P_{n+1},T_{n-1})=1$. Then, $B=\End_A(T)$ is given by the following quiver $Q_n'$
$$
\xymatrix @1 @R=0.3cm  @C=0.5cm{
1\ar[rr]_{\beta_1}&\ar@{.}@/^/[rr]&2\ar[rr]_{\beta_2}&\ar@{.}@/^/[rr]&3\ar[rr]_{\beta_3}&&4&\dots&\bullet\ar[rr]_{\beta_{n-2}}&\ar@{.}@/^/[rr]&n-1\ar[rr]_{\beta_{n-1}}&\ar@{.}@/^/[rr]&n\ar[rr]_{\beta_n}&&n+1\\
&&&&&&&&&&&&&&\\
b_1 \ar[uu]&&b_2 \ar[uu]&&b_3 \ar[uu]&&b_4 \ar[uu]&&b_{n-2} \ar[uu]&&b_{n-1} \ar[uu]&&&&
}$$
and $I'_n=<\beta_i\beta_{i+1}, \,\, i=1,\dots, n-1>$. Hence, by \cite{G} we get that  $\gd\,B=n$.
\end{ej}
\subsection{Special biserial algebras of global dimension two}
In this section, we start proving that the annihilator of a $\tau$-tilting $A$-module over a special biserial algebras is generated by paths.

\begin{lema}\label{SBanulador}
Let $A=kQ/I$ be a special biserial algebra and $T$ be a $\tau$-tilting $A$-module. Assume $0\neq p+q\in\emph{ann}\,T$, with $p,q$ non-zero paths of length  greater than or equal to two. Then $p\in\emph{ann}\,T$ and $q\in \emph{ann}\,T$.
 \end{lema}

\begin{proof}
Let $T=\bigoplus\limits_{i=1}^nT_i$ be a $\tau$-tilting $A$-module, with $T_i$ indecomposable $A$-modules, for $i=1,\dots,n$. Consider $p+q\in \ann\,T$, with $p,q$ non-zero paths in $Q$ from $a$ to $b$.  Assume that $p\notin \ann\, T$. Then, there exists $T_h=((T_h)_y,\varphi_\alpha)_{y\in Q_0, \alpha\in Q_1 }$ a direct summand of $T$ such that $p\notin \ann\, T_h$.  Since $p+q\in\ann\,T_h$ then  $q\notin \ann\,T_h$.

Let $S(x)$ be a direct summand of $\mbox{top}\, T_h$ such that there exists a non-zero path $\gamma$ from $x$ to  $a$ and $\varphi_{p}\varphi_{\gamma} \neq 0$. Then $\varphi_q\varphi_{\gamma}\neq 0$. Since $A$ is a special  biserial algebra if $\gamma p\neq 0 $ then $\gamma q =0$. Thus, $\varphi_q\varphi_{\gamma}= 0$ which is a contradiction. Hence, $S(a)$ is a direct summand of $\mbox{top}\, T_h$.

Let $\{m_{a_1},\ldots, m_{a_{d_a}}\}$ be  a basis for $(T_h)_a$. As in the proof of Proposition \ref{anulador}, we consider $P_1\stackrel{g}\rightarrow P_0\stackrel{f}\rightarrow T_h\rightarrow0$ a minimal projective presentation of $T_h$ such that $f(e_a)=m_{a_1}$ and $f(\epsilon)=\varphi_{\epsilon}(m_{a_1})$ if $\epsilon$ is a path in $Q$ starting in $a$. Then $ p+q \in \mbox{ker}f$, because
$$f(p+q)=\varphi_{p+q}(m_a)=(\varphi_p+\varphi_q)(m_a)=0.$$
Let us prove that $p+q \notin \mbox{rad}(\mbox{ker}f)$. Suppose that $p+q\in \mbox{rad}(\mbox{ker}\,f)$. Since the morphisms in the representation of $\mbox{ker}f$ are the multiplication  by arrows, then there exists $\nu_y\in (\mbox{ker} f)_b$ such that
\begin{equation}\label{eqnrelacion2}
p+q=\sum\limits_{\beta:y\rightarrow b}\nu_y\beta.
\end{equation}

The equation (\ref{eqnrelacion2}) define a relation in  $A$ which has more than three branches and ${p+q\neq0}$, which is a contradiction because $A$ is special biserial. Thus, ${p+q \notin \mbox{rad}(\mbox{ker}\,f)}$. Hence, $S(b)$ is a direct summand of  ${\mbox{top}\,(\mbox{ker}\,f)}$ and $P(b)$ is a direct summand of $P_1$.

As in the proof of the Proposition \ref{anulador}, if $P(b)$ is a direct summand of $P_1$ we can construct a morphism $\theta:P_1\rightarrow T_h$ which does not factor through  $g$ which is a contradiction. Therefore, if $p+q\in \ann\, T$ then $p\in \ann\, T$ and $q\in\ann\,T$.
\end{proof}

\begin{obs}
As a consequence of Lemma \ref{SBanulador}, we get that if $A$ is a special biserial algebra and $T$ is a $\tau$-tilting module (which is not tilting), then $\mbox{ann}\,T$ is generated by paths of length  greater than or equal to one.
\end{obs}

\begin{prop}
Let $A=kQ/I$ be a special biserial algebra such that $\emph{gldim}\,A=2$ and $T$ be a $\tau$-tilting  $A$-module. Then $A/\emph{ann}\,T$ satisfies $(GD1)$ and $(GD3)$, stated in Theorem \ref{SBgldim}.
\end{prop}
\begin{proof}
By Lemma \ref{SBanulador}, the annihilator of $T$ is generated by paths. Then, $A/\ann \,T$ has at least the same binomial relations that $A$. By Theorem \ref{SBgldim}, we have that $A$ satisfies $(GD1)$ and thus, $A/\ann \,T$ satisfies $(GD1)$.

Let $(\alpha p \beta, \gamma q \delta)$ be a binomial relation in  $A/\ann \,T$. Consider $u$ a non-zero path in $A/\ann \,T$ such that $e(u)=s(\alpha)$, $u\alpha p=0$, $u\gamma q=0 $ and minimal in the sense that if $u'$ is a subpath of $u$ then $u'\alpha p\neq 0$ or $u'\gamma q\neq0 $. We write  $u=\rho_1\dots\rho_k$ with $\rho_i$ an arrow for each $i=1,\dots,k$. Since  $A$ is special biserial then $\rho_k\alpha=0$ or $\rho_k\gamma=0$. Without loss of generality, assume that $\rho_k\alpha=0$. Then there exists a relation $\lambda$  in $A/\ann\, T$ contained in $u\gamma q$. Since $A$ satisfies $(GD3)$ then $\lambda$ is not  a relation in $A$. Thus, since $A$ satisfies $(GD1)$, by Proposition \ref{anulador} there exists a non-zero path $\lambda_1$ of length greater than or equal to one such that $\widetilde{\lambda}=\lambda_1\lambda$ is a zero-relation in  $A$.

Since $u\neq 0$ and $\gamma q \neq 0$ then $\lambda$ contains al least the path $u\gamma$ (it must contain $u$, because if not it is a contradiction to the minimality of $u$). Consider $\widetilde{u}=\lambda_1u$. Since $\widetilde{\lambda}$ is a zero-relation, then $\widetilde{u}\neq 0$. Hence, we have that $\widetilde{u}\alpha p=0$ in $A$ and $\widetilde{u}\gamma q = 0$ in $A$, which is a contradiction because $A$ satisfies $(GD3)$.

Now, if we consider $v$ a non-zero path in $A/\ann \,T$ such that ${s(v)=e(\beta)}$, with similar arguments as above it is not hard to see that if ${p\beta v=0}$ then  ${q\delta v=0}$. Therefore, $A/\ann\,T$ satisfies $(GD3)$.
\end{proof}

\begin{obs}\label{obsquiverproh}
Let $A=kQ/I$ be a special biserial algebra such that ${\mbox{gldim}\,A =2}$. The following quivers could not be subquivers of  $Q$.
\begin{enumerate}[a)]
\item A binomial relation $(\alpha p \gamma , \beta q \delta)$ such that $s(\alpha)=e(\gamma)$.
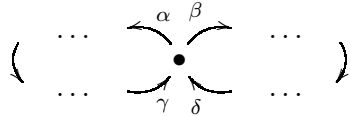
\[\xymatrix  @R=0.009cm  @C=0.09cm{
\ar@/_/[dd]&\dots&&&&&&\dots&\ar@/^/[dd]\\
&&&&\bullet\ar@/_/[llu]_{\alpha}\ar@/^/[rru]^{\beta} &&&&\\
&\dots&\ar@/_/[rru]_{\gamma}&&&&\ar@/^/[llu]^{\delta}&\dots&
}\]
If there exists such a binomial relation, since $A$ is special biserial, the following conditions are satisfied:
  \[(\delta\alpha=0  \,\,\, \mbox{or}  \,\,\, \delta\beta=0) \,\,\,\mbox{and} \,\,\,
  (\gamma\alpha=0  \,\,\,\mbox{or}  \,\,\, \gamma\beta=0) \,\,\,\mbox{and}\,\,\,
  (\gamma\alpha=0  \,\,\,\mbox{or} \,\,\, \delta\alpha=0) \,\,\,\mbox{and}\,\,\,
  (\gamma\beta=0  \,\,\,\mbox{or} \,\,\, \delta\beta=0).\]
Suppose that $\gamma\alpha=0$ and $\delta\beta=0$. Let $u=\alpha p' \gamma \neq 0$. Then $u\alpha p'=0$ and $u\beta q'=0$ which is a contradiction since $A$ satisfies $(GD3)$. A similar analysis hold if we consider the other cases.
\item A subquiver of the form:
$
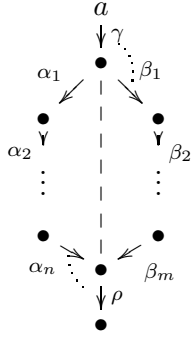
\xymatrix  @R=0.009cm  @C=0.01cm{
&& a\ar[dd]^{\gamma} &&&&\\
&&\ar@{.}@/^/[rdd] &&&&\\
&& \bullet \ar@{--}[dddddd] \ar[lldd]_{\alpha_1}\ar[rrdd]^{\beta_1} &&&&\\
&&& &&&&\\
\bullet \ar[dd]_{\alpha_2}&&&&\bullet\ar[dd]^{\beta_2}  &  \\
&&&&  &&&\\
\vdots &&&& \vdots &&&\\
 \bullet \ar[drr]_{\alpha_n\;\;\;}&\ar@{.}@/_/[ddr]&&& \bullet \ar[dll]^{\;\;\;\beta_m} &&\\
&& \bullet\ar[dd]^{\rho}  &&&\\
 &&   &&&\\
 &&\bullet   &&&
}$
\newline with $\gamma \beta_1=0$ and $\alpha_n \rho=0$. Indeed, considering $p_1=\gamma\alpha_1\dots\alpha_{n-1}$, $p_2=\alpha_n$ and $p_3=\rho$ we obtain that $p_1p_2$ and $p_2p_3$ are the only zero-relations contained in the path $p_1p_2p_3$, which is a contradiction since $A$ satisfies $(GD2)$.
\item The end-point of a zero-relation $\rho_1 \dots \rho_n$ in $A$ could not be the start-point of two arrows $\alpha,\,\beta$. Indeed, since $A$ is special biserial then $\rho_n\alpha=0$ or $\rho_n\beta=0$. Thus, taking the paths $p_1=\rho_1\dots \rho_{n-1}$, $p_2=\rho_n$ and $p_3=\alpha$ (or $p_3=\beta$), we get a contradiction since $A$ satisfies $(GD2)$.
\end{enumerate}
\end{obs}
In order to prove the main result of this section, first we give some lemmas.
\begin{lema}\label{lemabiserial1}
Let $A=kQ/I$ be a special biserial algebra such that ${\emph{gldim}\,A=2}$, $T$ be a $\tau$-tilting module and $a\in (Q_{A/\emph{ann}\,T})_0$. Assume that in the vertex $a$ only starts a binomial relation. Then $\emph{pd}_{\scriptscriptstyle A/\tiny{\emph{ann}}\,T}S(a)\leq 2$.
\end{lema}
\begin{proof}
Let $a\in (Q_{A/\tiny{\ann}\,T})_0$ be a vertex such that only starts a binomial relation $(\alpha p,\beta q)$ in $a$, with ${\alpha,\beta\in (Q_{A/\tiny{\ann}\,T})_1}$ and $p,q$ non-trivial paths.  The projective cover of $S(a)$ is $P(a)$ and ${\mbox{rad}\,P(a)=M(pq^{-1})}$.

We claim that $\pd_{\scriptscriptstyle A/\tiny{\ann}\,T}M(pq^{-1})\leq 1$. In fact, by \cite[Lemma 2.5]{HL}, it is enough to show that $pq^{-1}$ does not satisfy  $(PD1), \,(PD2),\,(PD3)$ and $(PD4)$. Since $A/\ann\,T$ satisfies $(GD1)$ and $(GD3)$, we have that $(PD4)$ and $(PD3)$ are not satisfy, respectively. If $(PD1)$ holds, then there exists an arrow $\lambda$ such that $\lambda^{-1}pq^{-1}$ is a reduced walk. Since $A$ is special biserial, there exists a zero-relation $\alpha\lambda$ starting in  $a$,  a contradiction. Similarly, we get that $(PD2)$ is not possible. Thus, $\pd_{\scriptscriptstyle A/\tiny{\ann}\,T}M(pq^{-1})\leq 1$. Therefore, $\pd_{\scriptscriptstyle A/\tiny{\ann}\,T} S(a)\leq 2$.
\end{proof}
\begin{lema}\label{SBcaso2}
Let $A=kQ/I$ be a special biserial algebra such that ${\emph{gldim}\,A=2}$, $T$ be a $\tau$-tilting module and $a\in (Q_{A/\emph{ann}\,T})_0$. Assume that $a$ is not the start-point of a binomial relation. Then $\emph{pd}_{\scriptscriptstyle A/\tiny{\emph{ann}}\,T}S(a)\leq 4$.
\end{lema}
\begin{proof}
Consider $a\in (Q_{{\scriptscriptstyle A/\tiny{\ann}\,T}})_0$ such that $a$ is not the start-point of a binomial relation and  $\pd_{\scriptscriptstyle A/\tiny{\ann}\,T}S(a)\geq 5$. Then we have a subquiver as follows:

$$
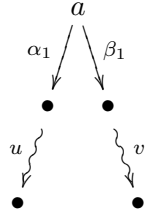
\xymatrix @!0 @R=1.3cm  @C=0.4cm{
&& a\ar[dr]^{\beta_1}\ar[dl]_{\alpha_1} &&\\
&\bullet\ar@{~>}[dl]_u&&\bullet\ar@{~>}[dr]^v&\\
\bullet &&&& \bullet
}$$
with $\alpha_1,\beta_1\in (Q_{\scriptscriptstyle A/\tiny{\ann}\,T})_1$ and  $u,v$ paths. Then the projective cover of $S(a)$ is $P(u^{-1}\alpha_1^{-1}\beta_1v)$. We have the following exact sequence:
\[0\rightarrow M(u)\oplus M(v) \rightarrow P(u^{-1}\alpha_1^{-1}\beta_1v)\rightarrow S(a)\rightarrow 0.\]
Since $\pd_{{\scriptscriptstyle A/\tiny{\ann}\,T}}S(a)\geq 5$, then we get that $\pd_{{\scriptscriptstyle A/\tiny{\ann}\,T}}M(u)\geq 4$ or $\pd_{{\scriptscriptstyle A/\tiny{\ann}\,T}}M(v)\geq 4$.  Without loss of generality, assume that $\pd_{{\scriptscriptstyle A/\tiny{\ann}\,T}}M(u) \geq 4$.

If $u$ is a trivial path then $e(\alpha_1)$ is not the start-point of a binomial relation, because $A/\ann\,T$ satisfies $(GD3)$.

Assume that $u$ is not a trivial path and  $e(\alpha_1)$ is the start-point of a binomial relation $(uu_2,\beta_2v_2)$, with ${\beta_2\in (Q_{{\scriptscriptstyle A/\tiny{\ann}\,T}})_1}$ and $u_2,\,v_2$ paths of length al least one. Observe that, since $A/\ann\,T$ satisfies $(GD3)$, $u_2$ is an arrow. Then, we have the following exact sequence:
\[0\rightarrow M(v_2)\rightarrow P(uu_2,\beta_2v_2)\rightarrow M(u)\rightarrow 0\]
where $P(uu_2,\beta_2v_2)$ is the projective cover of $M(u)$. Since $\pd_{{\scriptscriptstyle A/\tiny{\ann}\,T}}M(u)\geq 4$, $M(v_2)$ is not a projective module. Then, there exists $\rho\in (Q_{{\scriptscriptstyle A/\tiny{\ann}\,T}})_1$ such that either $\rho^{-1}v_2$ is a reduced walk or $v_2\rho$ is a reduced walk (observe that  $s(v_2)$ is not the start-point of a binomial relation because $A/\ann\,T$ satisfies $(GD1)$).

Assume that there exists $\rho\in (Q_{{\scriptscriptstyle A/\tiny{\ann}\,T}})_1$ such that $\rho^{-1}v_2$ is a reduced walk. Then since $\alpha_1u\neq 0$ in $A$ and $\beta_2v_2\neq 0$ in $A$ we get that $\alpha_1\beta_2=0$ in $A$ and $\beta_2\rho$ in $A$ which is a contradiction since $A$ satisfies $(GD2)$.
Then,  there exists $\rho\in (Q_{\scriptscriptstyle A/\tiny{\ann}\,T})_1$ such that $v_2\rho$ is a reduced walk. Since $v_2\rho \neq 0$ in $A$, then $u_2\rho$ must be zero in $A$. Since $\alpha_1\beta_2=0$ in $A$, by Remark \ref{obsquiverproh} we get a contradiction.

Therefore,  $e(\alpha_1)$ is not the start-point of a binomial relation. Since $\pd_{{\scriptscriptstyle A/\tiny{\ann}\,T}}M(u)\geq~4$, there exists either $\beta_2\in (Q_{{\scriptscriptstyle A/\tiny{\ann}\,T}})_1$ such that $\beta_2^{-1}u$ is a reduced walk or   $u_2'\in (Q_{{\scriptscriptstyle A/\tiny{\ann}\,T}})_1$ such that $uu_2'$ is a reduced walk.
with $v_1$ and $u_2$ paths. Then, we have the following exact sequence:
\[0\rightarrow M(v_1)\oplus M(u_2)\rightarrow P(v_1^{-1}\beta_2^{-1}uu'_2u_2)\rightarrow M(u)\rightarrow 0\]
where $P(v_1^{-1}\beta_2^{-1}uu'_2u_2)$ is the projective cover of $M(u)$ and $M(v_1)\neq 0$ or $M(u_2)\neq 0$. Without loss of generality, assume that $M(u_2)\neq 0$ and $\pd_{{\scriptscriptstyle A/\tiny{\ann}\,T}}M(u_2)\geq 3$. Observe that, in this case, there exists a zero-relation $r_1=\alpha_1uu'_2$ in $A/\ann\,T$ ($u$ could be a trivial path).

According to Lemma \ref{presentacion}, we have three possibilities: $r_1$ is a relation in $A$, there exists a non trivial path $\gamma$ such that $\gamma r_1$ is a zero-relation in $A$ or there exist  paths $\rho, p$ such that $(\rho r_1, p)$ is a binomial relation in $A$ (observe that $\rho$ could be a trivial path).

Assume that $r_1$ is a zero-relation in $A$. Then $e(u'_2)$ is not the start-point of two arrows. Since $M(u_2)$ is not a projective module, there exists $u'_3\in (Q_{{\scriptscriptstyle A/\tiny{\ann}\,T}})_1$ such that $u_2u'_3$ is a reduced walk in $A/\ann\,T$. We have the following exact sequence:
\[0\rightarrow M(u_3)\rightarrow P(u_2u_3'u_3)\rightarrow M(u_2)\rightarrow 0\]
where $P(u_2u_3'u_3)$ is the projective cover of $M(u_2)$.  Then there exists a zero-relation $r_2$ contained in the path $uu'_2u_2u'_3$ that contains the path $u'_2u_2u'_3$. Let $r_2=\widetilde{u}u_2'u_2u_3'$, with $u=\widehat{u}\,\widetilde{u}$. Since $\gd\,A=2$, then $r_2$ is not a zero-relation in $A$.

Assume that there exists a non trivial path $\gamma$ such that $\gamma r_2$ is a zero-relation in $A$. Since $r_2$ is of length al least two, then $\gamma r_2$ is of length at least three. Therefore, since $A$ is special biserial $u$ is trivial and $r_1$ is of length two. Then $\alpha_1u_2'$ and $\gamma u_2'u_2u_3'$ are zero-relations in $A$. Since $e(u_3')$ is the end-point of a zero-relation in $A$, then $e(u_3')$ is not be the start-point of two arrows. Since $M(u_3)$ is not a projective module, there exists $u'_4\in (Q_{{\scriptscriptstyle A/\tiny{\ann}\,T}})_1$ such that $u_3u_4'$ is a reduced walk. Then we have a relation $r_3$ in $A/\ann\,T$ contained in the path $u_2u_3'u_3u_4'$ that contains the path $u_3'u_3u_4'$. Let $r_2=\widetilde{u_2}u_3'u_3u_4'$, with $u_2=\widehat{u_2}\,\widetilde{u_2}$. Since $A$ satisfies $(GD2)$, then $r_3$ is not a relation in $A$. Since $\gamma r_2$ is of length al least three, then there is no path $\gamma'$ such that $\gamma'r_3$ is a zero-relation in $A$ because $A$ is special biserial. Assume that there exists a path $\theta$ such that $(\theta r_3,\lambda p)$ is a binomial relation in $A$, with $\lambda$ an arrow and $p$ a path. Then $\theta$ is a subpath of $u_2'\widehat{u_2}$. Since $u_2'u_2u_3'\neq 0$, then $u_2'\widehat{u_2}\lambda$ is zero because $A$ is special biserial. Since $\gamma u_2'u_2u_3' =0$ we get a contradiction because $A$ satisfies $(GD3)$. Hence, for this case, $M(u_3)$ is projective and $\pd_{\scriptscriptstyle A/\tiny{\ann}\,T}S(a)\leq 3$.

Assume that $r_2$ is related with a binomial relation in $A$. Then we have two possibilities: $(r_2,\lambda_1 p \lambda_2)$ is a binomial relation in $A$ or there exists a non-trivial path $\gamma$ such that $(\gamma r_2,\lambda_1 p \lambda_2)$ is a binomial relation in $A$.  First, suppose  that $(r_2,\lambda_1 p \lambda_2)$ is a binomial relation in $A$, with $\lambda_1, \lambda_2$ arrows and $p$ a path. If $u$ is not a trivial path, then since $uu_2'\neq 0$ we get that $u\lambda=0$. Taking the path $\alpha_1\widehat{u}$ we have a contradiction because $A$ satisfies $(GD3)$. Therefore, $u$ is a trivial path and we have that $r_1=\alpha_1 u_2'$ and $r_2=u_2'u_2u_3'$. By Remark \ref{obsquiverproh}, $e(u_3')$ is not the start-point of two arrows. Since $M(u_3)$ is not a projective module, then there exists $u_4'\in (Q_{{\scriptscriptstyle A/\tiny{\ann}\,T}})_1$ such that $u_3u_4'$ is a reduced walk. Then we have a zero-relation $r_3$ in $A/\ann\,T$ contained in the path $u_2u_3'u_3u_4'$ that contains the path $u_3'u_3u_4'$. Let $r_3=\widetilde{u_2}u_3'u_3u_4'$, with $u_2=\widehat{u_2}\,\widetilde{u_2}$. Since $A$ satisfies $(GD3)$ and $(GD1)$, $r_3$ is a zero-relation in $A$. Moreover, since $A$ satisfies $(GD3)$, $r_3$ is of length two and then $\widetilde{u_2}$ and $u_3$ are trivial paths. Then $r_3=u_3'u_4'$. Since $u_3$ is a trivial path, $M(u_3)$ is a simple module. Then we have the  exact sequence:
\[0\rightarrow M(u_4)\rightarrow P(u_3u_4'u_4)\rightarrow M(u_3)\rightarrow 0\]
where $P(u_3u_4'u_4)$ is the projective cover of $M(u_3)$. Since $e(u_4')$ is the end-point of a zero-relation in $A$, then $e(u_4')$ is not the start-point of two arrows. Since $M(u_4)$ is not a projective module, there exists $u_5'\in (Q_{{\scriptscriptstyle A/\tiny{\ann}\,T}})_1$ such that $u_4u_5'$ is a reduced walk in $A/\ann\,T$. Then, there exists a zero-relation $r_4=u_4'u_4u_5'$ in $A/\ann\,T$. Since $\gd\,A=2$, then $r_4$ is not a zero-relation in $A$. Since $A$ satisfies $(GD3)$ and $(GD1)$, then $r_4$ is not  related neither with a zero-relation in $A$ nor a binomial relation in $A$, which is a contradiction. Therefore, $M(u_4)$ is a projective module and $\pd_{\scriptscriptstyle A/\tiny{\ann}\,T}S(a)\leq 4$.

Finally, assume that there exists a non-trivial path $\gamma$ such that $(\gamma r_2, \lambda_1 p \lambda_2)$ is a binomial relation in $A$. Since $A$ is special biserial and $\gamma r_2\neq 0$, we get that $u$ is a trivial path and $r_1=\alpha_1u_2'$. Clearly, $e(u_3')$ is not the start-point of a binomial relation because $A$ satisfies $(GD1)$. Since $M(u_3)$ is not a projective module there exists either $\beta \in (Q_{{\scriptscriptstyle A/\tiny{\ann}\,T}})_1$ such that $\beta_4^{-1}u_3$ is a reduced walk or $u_4'\in (Q_{{\scriptscriptstyle A/\tiny{\ann}\,T}})_1$ such that $u_3u_4'$ is a reduced walk. Suppose that there exists $\beta$ such that $\beta u_3$ is a reduced walk. Then there exists a zero-relation $r_3$ contained in the path $u_2u_3'\beta$ that contains the path $u_3'\beta$. Since $A$ satisfies $(GD1)$ and $(GD3)$, $r_3$ is a zero-relation in $A$ of length two. Then $r_3=u_3'\beta$. Moreover, for this case there is no $u_4'$ such that $u_3u_4'$ is a reduced walk in $A/\ann\,T$ because $A$ satisfies $(GD1)$ and $(GD3)$. Then we get the following exact sequence:
\[0\rightarrow M(v)\rightarrow P(v^{-1}\beta^{-1}u_3)\rightarrow M(u_3)\rightarrow 0.\]
Since $r_3$ is a zero-relation in $A$, then $e(\beta)$ is not the start-point of two arrows. Since $M(v)$ is not a projective module, there exists $u_5'\in (Q_{{\scriptscriptstyle A/\tiny{\ann}\,T}})_1$ such that $vu_5'$ is a reduced walk. Then there exists a zero-relation $r_4=\beta v u_5'$. The relation $r_4$ is not a relation in $A$ since $\gd\,A=2$. On the other hand, since $A$ satisfies $(GD3)$ and $(GD1)$, $r_4$ is neither related with a zero-relation in $A$ nor with a binomial relation in $A$. Therefore, $M(v)$ is projective and $\pd_{\scriptscriptstyle A/\tiny{\ann}\,T}S(a)\leq 4$. Similarly, if there exists $u_4'\in (Q_{{\scriptscriptstyle A/\tiny{\ann}\,T}})_1$ such that $u_3u_4'$ is a reduced walk. Hence, if $r_1$ is a zero-relation in $A$ then $\pd_{\scriptscriptstyle A/\tiny{\ann}\,T}M(u_2)\leq 2$.

In  the other cases for $r_1$, a similar analysis as above, allow us to get the result. Therefore, ${\pd_{\scriptscriptstyle A/\tiny{\ann}\,T}S(a)\leq 4}$.
\end{proof}

\begin{lema}\label{lemabiserial2}
Let $A=kQ/I$ be a special biserial algebra such that ${\emph{gldim}\,A=2}$, $T$ be a $\tau$-tilting $A$-module and $a\in (Q_{A/\emph{ann}\,T})_0$ such that $a$  is the start-point of a binomial relation and the start-point of a zero-relation. Then $\emph{pd}_{A/\emph{ann}\,T}S(a)\leq 4$.
\end{lema}
\begin{proof}
Let $a\in (Q_{\scriptscriptstyle A/\tiny{\ann}\,T})_0$ such that  $a$ is the start-point of a binomial relation $(\alpha_1 p_1,\beta_1 q_1)$, with $\alpha,\beta\in (Q_{\scriptscriptstyle A/\tiny{\ann}\,T})_1$ and $p,q$ non-trivial paths and where $a$ is also the start point of a zero-relation. Assume that $\pd_{\scriptscriptstyle A/\tiny{\ann}\,T}S(a)\geq 5$. Then, we have a subquiver as follows:
$$
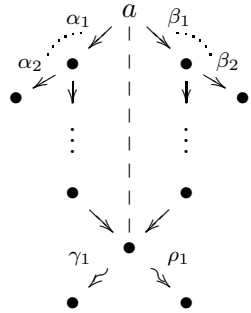
\xymatrix  @R=0.005cm  @C=0.005cm{
&&&& a\ar[ddll]_{\alpha_1}\ar[ddrr]^{\beta_1}\ar@{--}[ddddddd] &&&&\\
&&&\ar@{.}@/_/[dll]&&\ar@{.}@/^/[drr]&&&\\
&&\bullet\ar[dd]\ar[dll]_{\alpha_2} &&&&\bullet\ar[dd]\ar[drr]^{\beta_2} &&\\
\bullet&&&&&&&&\bullet\\
&&\vdots&&&&\vdots&&\\
&&\bullet \ar[ddrr]&&&&\bullet\ar[ddll]&&\\
&&&&&&&&\\
&&&&\bullet \ar@{~>}[ddll]_{\gamma_1}\ar@{~>}[ddrr]^{\rho_1}&&&&\\
&&&&&&&\\
&&\bullet&&&&\bullet&&
}$$
with $\alpha_2,\,\beta_2\in (Q_{{\scriptscriptstyle A/\tiny{\ann}\,T}})_1$ and $\gamma_1,\,\rho_1$ paths. Note that ${\mbox{rad}\,P(a)=M(p_1q_1^{-1})}$. Then, we get the following exact sequence:
\[0\rightarrow M(p_2)\oplus M(\gamma_1^{-1}\rho_1)\oplus M(q_2)\rightarrow P(p_2^{-1}\alpha_2^{-1}p_1\gamma_1)\oplus P(q_2^{-1}\beta_2^{-1}q_1\rho_1)\rightarrow M(p_1q_1^{-1})\rightarrow 0\]
with $\alpha_2,\,\beta_2 \in (Q_{{\scriptscriptstyle A/\tiny{\ann}\,T}})_1$, $q_2,p_2,\rho_1,\gamma_1$ paths and ${P(p_2^{-1}\alpha_2^{-1}p_1\gamma_1)\oplus P(q_2^{-1}\beta_2^{-1}q_1\rho_1)}$ the projective cover of $M(p_1q_1^{-1})$. Since $A/\ann\,T$ satisfies $(GD3)$, then $M(\gamma^{-1}\rho_1)$ is projective. Thus, since ${\pd_{\scriptscriptstyle A/\tiny{\ann}\,T}(M(p_1q_1^{-1}))\geq 4}$, either ${\pd_{\scriptscriptstyle A/\tiny{\ann}\,T}M(p_2)\geq 3}$ or ${\pd_{\scriptscriptstyle A/\tiny{\ann}\,T}M(q_2)\geq 3}$.

Without loss of generality, suppose that $\pd_{\scriptscriptstyle A/\tiny{\ann}\,T}M(p_2)\geq 3$. Then we have a zero-relation  $r_1=\alpha_1\alpha_2$ in $A$. Then $e(\alpha_2)$ is not the start-point of two arrows. Since $M(p_2)$ is not a projective module there exists $\alpha_3\in (Q_{{\scriptscriptstyle A/\tiny{\ann}\,T}})_1$ such that $p_2\alpha_3$ is a reduced walk  in $A/\ann\,T$. We have the following exact sequence:
\[0\rightarrow M(p_3) \rightarrow P(p_2\alpha_3p_3)\rightarrow M(p_2)\rightarrow 0\]
where $P(p_2\alpha_3p_3)$ is the projective cover of $M(p_2)$.  Then there exists a zero-relation ${r_2=\alpha_2p_2\alpha_3}$ in $A/\ann\,T$. Since $A$ satisfies $(GD2)$, then $r_2$ is not a zero-relation in $A$. Since $A$ satisfies $(GD1)$, then $r_2$ is not a maximal subpath of a binomial relation in $A$.

Assume that there exists a non-trivial path $\gamma$ such that $\gamma r_2$ is a zero-relation in $A$. Since $e(\alpha_3)$ is the end-point of a zero-relation in $A$, then $e(\alpha_3)$is not the start-point of two arrows. Since $M(p_3)$ is not a projective module, there exists $\alpha_4\in (Q_{{\scriptscriptstyle A/\tiny{\ann}\,T}})_1$ such that $p_3\alpha_4$ is a reduced walk. Then there exists a zero-relation $r_3$ in $A/\ann\,T$ contained in the path $p_2\alpha_3p_3\alpha_4$ that contains the path $\alpha_3p_3\alpha_4$. Let $r_3=\widetilde{p_2}\alpha_3p_3\alpha_4$, with $p_2=\widehat{p_2}\,\widetilde{p_2}$. Since $\gd\,A=2$, then  $r_3$ is not a zero-relation in $A$. Suppose that there exists a non-trivial path $\rho$ such that $\rho r_3$ is a zero-relation in $A$. Then $\rho r_3$ is a zero-relation in $A$ of length at least three and $\alpha_2 p_2 \alpha_3 \neq 0 $,  a contradiction since $A$ is special biserial. Next, assume that $(r_3,\lambda_1 q \lambda_2)$ is a binomial relation in $A$. Since $\alpha_2\widehat{p_2}r_3\neq 0$ in $A$ we get that $\alpha_2\widehat{p_2}\lambda_1$ is zero in $A$. Therefore, considering the path $\gamma \alpha_2 \widehat{p_2}$ we get a contradiction because $A$ satisfies $(GD3)$. Finally, assume that there exists a non-trivial path $\theta$ such that $(\theta r_3, \lambda_1 q \lambda_2)$ is a binomial relation in $A$. Then $\theta$ is a subpath of $\gamma \alpha_2 p_2$ and we get a contradiction since $A$ satisfies $(GD3)$. Therefore, $M(p_3)$ is a projective module and $\pd_{\scriptscriptstyle A/\tiny{\ann}\,T}M(p_2)\leq 1$.

Finally, assume that there exists  a non-trivial path $\gamma$ such that $(\gamma r_2, \lambda_1 q \lambda_2)$ is a binomial relation in $A$. Then with similar analysis as we did in the proof of Lemma \ref{SBcaso2} we conclude that $\pd_{\scriptscriptstyle A/\tiny{\ann}\,T}M(p_2)\leq 2$.

Therefore, $\pd_{\scriptscriptstyle A/\tiny{\ann}\,T}S(a)\leq 4$.
\end{proof}

As an immediate consequence of  Lemma \ref{lemabiserial1}, Lemma \ref{SBcaso2} and  Lemma \ref{lemabiserial2} we get the main result of this section.
\begin{teo}\label{resultado principal biserial}
Let $A=kQ/I$ be a special biserial algebra such that ${\emph{gldim}\,A=2}$ and $T$ be a $\tau$-tilting $A$-module. Then $\emph{gldim}\,A/\emph{ann}\,T\leq 4$.
\end{teo}

The following example shows that the bound given by Theorem \ref{resultado principal biserial} for the global dimension of $A/\ann\,T$ is minimum.

\begin{ej}
Let $A=kQ/I$ be the following special biserial algebra:
$$
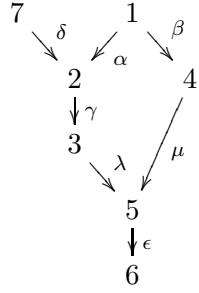
\xymatrix @R=0.35cm  @C=0.3cm{
7 \ar[rd]^{\delta} & & 1\ar[rd]^{\beta}\ar[ld]^{\alpha}& \\
 & 2\ar[d]^{\gamma}& & 4\ar[ddl]^{\mu}\\
 & 3 \ar[dr]^{\lambda} & & \\
 & & 5\ar[d]^{\epsilon} &\\
 & & 6 &
}$$
where $I=<\delta\gamma, \lambda\epsilon, \alpha\gamma\lambda-\beta\mu>$. It is easy to see that $\gd\,A=2$. Consider ${T=1 \oplus 6 \oplus {\small{\txt{1\\2\\3}}} \oplus {\small{\txt{1\\4}}}\oplus {\small{\txt{3\\5}}} \oplus {\small{\txt{5\\6}}}\oplus {\small{\txt{1\;\; 7\\2}}}}$. The annihilator of $T$ is generated by the path $\gamma \lambda$. Then $A/\ann\,T$ is given by the same quiver of $A$ with relations $I'=<\delta\gamma, \lambda\epsilon, \gamma\lambda, \beta\mu>$. It is not hard to see that $\pd_{A/\ann\,T}S(7)=4$ and $\gd\,A/\ann\,T =4$.
\end{ej}

As a direct consequence of the above theorem we have Theorem D.

\begin{cor}
Let $A=kQ/I$ be a special biserial algebra such that ${\emph{gldim}\,A=2}$, $T$ be a $\tau$-tilting $A$-module and $B=\emph{End}_AT$. Then $\emph{gldim}\,B\leq 5$.
\end{cor}


\begin{thebibliography}{Dillo 83}
\bibitem[AIR]{AIR} Adachi, T., Iyama, O. y  Reiten, I.; \textit{$\tau$-tilting theory}. Compositio Mathematica, 150(03), 415-452, 2014.
\bibitem[ASS]{ASS}  Assem, I., Simson, D. and  Skowronski,A.; {\it Elements of the representation theory of associative algebras}. London Math. Soc. Student Texts 65. Cambridge University Press, 2006.

\bibitem[AS]{AS} Auslander, M. and  Smalø, S. O.; \textit{Almost split sequences in subcategories}. Journal of Algebra, 69(2), 426-454, 1981.

\bibitem[G]{G} Gauvreau, C.; \textit{A new proof of a theorem of Green, Happel and Zacharia}.  Annales des sciences mathématiques du Québec. Université du Québec à Montréal, Département de mathématiques et informatique, 11 (01), 83-89. 1997.
\bibitem[GHZ]{GHZ} Green, E. L., Happel, D. and Zacharia, D.; \textit{Projective resolutions over Artin algebras with zero-relations.} Illinois Journal of Mathematics, 29(01), 180-190. 1985.
\bibitem[HL]{HL}Huard, F. y Liu, S.;  \textit{Tilted special biserial algebras}. Journal of Algebra, 217(02), 679-700. 1999.
\bibitem[IT]{IT}  Ingalls,C y Thomas, H.; \textit{Noncrossing partitions and representations of quivers}. Compositio Math., 145(06), 1533-1562, 2009.
\bibitem[M]{Miya} Miyashita, Y.; \textit{Tilting modules of finite projective dimension}. Mathematische Zeitschrift, 193(01), 113-146, 1986.
\end{thebibliography}
\end{document}